\documentclass[10pt, a4paper]{article}

\usepackage{amsmath, amssymb, amsfonts, amsthm}
\usepackage{times}
\usepackage{epsfig}
\usepackage{graphicx}
\usepackage{graphics}
\usepackage{geometry}
\usepackage{wrapfig}
\usepackage{mathrsfs}
\usepackage{mathtools}
\usepackage{amssymb}
\usepackage{dsfont}
\usepackage{xspace}
\usepackage{color}
\usepackage{enumitem}
\usepackage{subfigure}
\usepackage{thmtools}
\usepackage{tikz}
\usepackage{lscape}
\usepackage{wrapfig}
\usepackage{placeins}
\usepackage{url}

\usepackage{graphpap}
\usepackage{dsfont}
\usepackage{array}
\usepackage{pgfplots, pgfplotstable}

\newcommand{\notinclude}[1]{}

\newcommand{\domain}{{D}}
\newcommand{\W}{{\mathcal{W}}}
\newcommand{\Wd}{{\mathcal{W}_{\delta}}}
\newcommand{\Wdg}{{\mathcal{W}_{\delta,\gamma}}}
\newcommand{\WdgK}{{\mathcal{W}^K_{\delta,\gamma}}}
\newcommand{\WdgKh}{{\mathcal{W}^K_{\delta,\gamma,h}}}
\newcommand{\E}{{\mathcal{E}}}
\newcommand{\EK}{{\mathbf{E}^K}}

\newcommand{\Edg}{{\mathcal{E}}_{\delta,\gamma}}
\newcommand{\Edgd}{{\mathcal{E}}^D_{\delta,\gamma}}
\newcommand{\EdgK}{{\mathbf{E}}^K_{\delta,\gamma}}
\newcommand{\EdgKd}{{\mathbf{E}}^{K,D}_{\delta,\gamma}}
\newcommand{\EdgKh}{{\mathbf{E}}^K_{\delta,\gamma,h}}
\newcommand{\EdgKdh}{{\mathbf{E}}^{K,D}_{\delta,\gamma,h}}

\newcommand{\G}{{\mathcal{G}}}

\newcommand{\Gdg}{{\mathcal{G}_{\delta,\gamma}}}
\newcommand{\F}{{\mathbf{F}}}
\newcommand{\Fd}{{\mathbf{F}^D}}
\newcommand{\Fdt}{{\mathbf{F}_{\mbox{\tiny transport}}^D}}
\newcommand{\Fdz}{{\mathbf{F}_{\mbox{\tiny source}}^D}}
\newcommand{\Fdv}{{\mathbf{F}_{\mbox{\tiny viscous}}^D}}

\newcommand{\bary}{\lambda}
\newcommand{\Bdg}{{\mathcal{B}}_{\delta,\gamma}}

\newcommand{\BdgKh}{{\mathbf{B}}^K_{\delta,\gamma,h}}

\newcommand\restr[2]{{\left.\kern-\nulldelimiterspace #1 \vphantom{\big|}\right|_{#2}}}

\newcommand{\beq}{\begin{equation*}}
\newcommand{\eeq}{\end{equation*}}
\newcommand{\beqn}{\begin{equation}}
\newcommand{\eeqn}{\end{equation}}
\newcommand{\beqa}{\begin{eqnarray*}}
\newcommand{\eeqa}{\end{eqnarray*}}
\newcommand{\beqan}{\begin{eqnarray}}
\newcommand{\eeqan}{\end{eqnarray}}

\newcounter{counter}

\makeatletter
\DeclareRobustCommand\onedot{\futurelet\@let@token\@onedot}
\def\@onedot{\ifx\@let@token.\else.\null\fi\xspace}

\def\eg{\emph{e.g}\onedot} 
\def\ie{\emph{i.e}\onedot}

\renewcommand{\u}{{\theta}} \newcommand{\U}{{\Theta}}    \newcommand{\Uinnervec}{{ \mathbf{\Theta}}} 
\newcommand{\N}{\mathbb{N}}

\newcommand{\R}{\mathbb{R}}

\newcommand{\diff}{{\,\mathrm{d}}}
\newcommand{\id}{\mathrm{1\!I}}
\newcommand{\Id}{\mathrm{1\!I}}

\renewcommand{\div}{{\mathrm{div}}}

\newcommand{\tr}{{\mathrm{tr}}}

\newcommand{\penaltyPushforward}{\delta}
\newcommand{\penaltyHigherOrderDerivative}{\epsilon}

\newcommand{\Imagespace}{\mathcal{I}} 
\newcommand{\deformationSpace}{\mathcal{A}} 
\newcommand{\Wreg}{W}

\def\FG{}

\def\bfv{{\FG v}} \def\bfw{{\FG w}}

\def\bfnu{{\nu}} 
\def\dd{\;\!\mathrm{d}}
\newcommand{\OMeas}{\mathscr{M}^+(\Omega)}
\newcommand{\OSMeas}{\mathscr{M}(\Omega)}
\newcommand{\OSMeasd}{\mathscr{M}(\Omega;\R^d)}
\newcommand{\Meas}{\mathscr{M}^+(D)}
\newcommand{\SMeas}{\mathscr{M}(D)}
\newcommand{\SMeasd}{\mathscr{M}(D;\R^d)}
\newcommand{\STMeas}{\mathscr{M}([0,1] \times D)}
\newcommand{\STMeasp}{\mathscr{M_+}([0,1] \times D)}
\newcommand{\STMeasd}{\mathscr{M}([0,1] \times D;\R^d)}
\newcommand{\CE}{\mathcal{CE}}
\newcommand{\cA}{\mathcal{D}}
\newcommand{\weakly}{\rightharpoonup}
\newcommand{\Leb}{\mathscr{L}}
\newcommand{\ABB}{\cA_{\text{BB}}}
\newcommand{\AM}{\cA_{Z}}

\newcommand{\basisfct}{\xi}

\newcommand{\V}{ {\mathcal{V}} }
\newcommand{\mass}{{\mathbf{M}}}
\newcommand{\MatrixBF}{\mathbf{A}}
\newcommand{\RBF}{\mathbf{R}}
\newcommand{\SBF}{\mathbf{B}}
\newcommand{\TBF}{\mathbf{T}}

\newtheorem{Definition}{Definition}[section]
\newtheorem{Theorem}[Definition]{Theorem}
\newtheorem{Proposition}[Definition]{Proposition}
\newtheorem{Lemma}[Definition]{Lemma}

\definecolor{darkblue}{RGB}{0,0,160}

\pgfplotsset{compat=1.5}

\begin{document}

\title{A generalized model for optimal transport of images including dissipation and density modulation}
\author{Jan Maas, Martin Rumpf, Carola Sch\"onlieb, Stefan Simon}
\maketitle
\begin{abstract} 
In this paper the optimal transport and the metamorphosis perspectives are combined. For a pair of given input images geodesic paths in the space of images are defined as minimizers of a resulting path energy. To this end, the underlying Riemannian metric measures the rate of transport cost  and the rate of viscous dissipation. Furthermore, the model is capable to deal with strongly varying image contrast and explicitly allows for sources and sinks in the transport equations which are incorporated in the metric related to the metamorphosis approach by Trouv{\'e} and Younes. 
In the non-viscous case with source term existence of geodesic paths is proven in the space of measures.
The proposed model is explored on the range from merely optimal transport to strongly dissipative dynamics. 
For this model a robust and effective variational time discretization of geodesic paths is proposed. 
This requires to minimize a discrete path energy consisting of a sum of consecutive image matching functionals.
These functionals are defined on corresponding pairs of intensity functions and on associated pairwise matching deformations. 
Existence of time discrete geodesics is demonstrated. Furthermore, a finite element implementation is proposed and applied to instructive test cases and to real images. 
In the non-viscous case this is 
compared to the algorithm proposed by Benamou and Brenier including a discretization of the source term.
Finally, the model is generalized to define discrete weighted barycentres with applications to textures and objects. 
\end{abstract}

\section*{Introduction}
In the past two decades concepts from finite dimensional classical geometry have been successfully transferred to infinite-dimensional spaces, where shapes are contour curves of geometric objects, surfaces, image intensity maps, or  probability densities. These concepts have a continuously increasing impact on the development of novel computational tools in computer vision and imaging, ranging from shape morphing and modeling \cite{KiMiPo07}, and shape statistics, \eg
\cite{FlLuPi04}, to texture analysis \cite{rabin2012wasserstein} and computational anatomy \cite{BeMiTrYo02}. 
Three particularly influential approaches on the space of image maps are linked to optimal transportation \cite{monge1781memoire, BeBr00, ZhYaHa07,papadakis2014optimal}, the flow of diffeomorphism \cite{Ar66a, DuGrMi98} 
and metamorphosis \cite{MiYo01,TrYo05}. Here, we combine these three approaches and explore properties of
the resulting image manifold. Thus, in what follows we briefly review the underlying concepts. 

{\bf Optimal transport and its application in imaging.}
The problem of optimal transport is introduced in the seminal work of Monge \cite{monge1781memoire} in 1781. 
In \cite{Kan42,kantorovich2004problem} Kantorovich proposes a relaxed formulation of Monge's problem which gives rise to the Wasserstein distance considered in this paper. Let $(\domain, d)$ constitute a metric space. The $2$-Wasserstein distance between two probability measures $\mu_A,\mu_B\in\mathbb{P}(\domain)$ is defined by
\begin{equation}\label{pwasserdist}
\W(\mu_A,\mu_B)^2:=\min_{\pi\in\Gamma(\mu_A,\mu_B)} \int_{\domain \times \domain} d(x,y)^2 \diff \pi(x,y).
\end{equation}
Here $\Gamma(\mu_A,\mu_B)$ denotes the set of all probability measures on $\domain \times \domain$
with marginals $\mu_A$ and $\mu_B$ with respect to $x$ and $y$, respectively. 
For an introduction to optimal transport and the Wasserstein distance we refer the reader to the reviews \cite{1047.35001,0954.35011,cedric2003topics,ambrosio2006gradient,villani2008optimal}.

Being a distance function applicable to very general measures (continuous and discrete measures)
the Wasserstein distance has an increasing impact on robust distance measures in imaging \cite{rubner2000earth,1109.49027,peyre2012wasserstein,burger2012regularized}. As the Wasserstein metric is defined for arbitrary probability measures with finite second moment it allows to measure distances between absolutely continuous measures with respect to the Lebesgue measure as well as concentrated measures. With the increase of the complexity of applications efficient numerical computation of \eqref{mapform} became increasingly important. In that respect Benamou and Brenier propose an alternative formulation of the quadratic Wasserstein distance using the perspective of the underlying flow of a density $\u$ with Eulerian velocity $v$ and expressing the transport in terms of a constrained flow \cite{BeBr00}. That is, one asks for a minimizer of the path energy 
\beqn\label{eq:BB}
\E[\u, v] = \int_0^1 \int_\domain \u |v|^2 \diff t
\eeqn
for a density function $\u: [0,1] \times \domain \to \R$ and a velocity field $v: [0,1] \times \domain \to \R^d$
subject to the transport equation $\partial_t \u + \div (\u v) = 0$ and 
the constraints $\u(0) = \u_A$ and $\u(1) = \u_B$. Then the minimal energy is indeed the squared Wasserstein distance between measures $\mu_A$ and $\mu_B$ with corresponding densities $\u_A$ and $\u_B$ respectively. Their algorithm has been immensely influential in the numerical computation of the Wasserstein distance and gradient flows related to it, see \eg \cite{1194.35026,during2010gradient,papadakis2014optimal}. In \cite{papadakis2014optimal}, for instance, a proximal point algorithm for the solution of Benamou-Brenier's formulation \eqref{eq:BB} is derived and applied for the computation of Wasserstein geodesics between two image densities. Alternatively, if $\domain \subseteq \R^d$ is a strictly convex domain, $d$ is the Euclidean distance, and $\mu_A$ and $\mu_B$ are absolutely continuous measures with densities $\u_A$ and $\u_B$, respectively, then one has 
\begin{equation}\label{mapform}
\W(\mu_A,\mu_B)^2 = \min_{\phi_{\#} \mu_A = \mu_B} \int_\domain |\phi(x)-x|^2 \u_A(x) \diff x \,,
\end{equation}
where $\phi_{\#} \mu_A$ denotes the push forward of the measure $\mu_A$ under the mapping $\phi$
and for a diffeomorphism $\phi$ the constraint can be expressed as $(\det D\phi) \, \u_B \circ \phi = \u_A$. In \cite{angenent2003minimizing} an initial mass preserving transport map \cite{moser1965volume} is created and then an explicit time stepping scheme is employed to compute the optimal map from a modified formulation of \eqref{mapform} where the constraint is linearized. In \cite{haber2010efficient} the authors pick up formulation \eqref{mapform} as well and use a sequential quadratic programming method for its optimisation. Moreover, in \cite{chartrand2009gradient} the authors propose a gradient descent for the dual formulation of \eqref{mapform} and show its use for image registration and warping. In \cite{loeper2005numerical,saumier2010efficient} a damped Newton method is used to compute a solution $\Psi$ of the Monge-Ampere 
equation (which is the equality constraint on $\phi=\nabla\Psi$) and subsequently the optimal transport map $\phi$. Finally, let us mention that in \cite{schmitzer2013hierarchical} the authors propose another interesting numerical algorithm for the efficient computation of the Wasserstein distance that is based on an extension of the auction algorithm. The latter optimizes the cost functional in the Wasserstein distance only on a sparse subset of possible assignment pairs still guaranteeing global optimality. Various alternative computational approaches for the Wasserstein distance exist, \eg \cite{dean2006numerical,oberman2008wide,papadakis2014optimal}. 

In terms of imaging application the Wasserstein distance has been employed in the context of image and shape classification, segmentation, registration and warping, and image smoothing. In \cite{rubner2000earth} the Wasserstein distance is used as a distance on the images directly, interpreting images as discrete measures. In other approaches the Wasserstein distance is used on image histograms or points clouds (which could be feature vectors of images), see e.g. \cite{grauman2004fast,ling2007efficient}. In the context of image segmentation and classification similar approaches are used, see, \eg \cite{chan2007histogram,ni2009local,peyre2012wasserstein,oudre2012classification}. In \cite{rabin2012wasserstein} a discrete Wasserstein distance is employed for the computation of barycentres of discrete probability distributions with applications to texture synthesis and mixing. Thereby the original Wasserstein metric is replaced by a sliced version over one-dimensional distributions. In \cite{rabin2010geodesic} 
the same approach is used on clouds of geodesic shape descriptors as a similarity measure to discriminate between different shapes in 2D and 3D shape retrieval. The Wasserstein distance is also used in the context of contrast and colour modification in, \eg \cite{rabin2011wasserstein,ferradans2013regularized}.
In \cite{zhu2003area,haker2004optimal} the quadratic Wasserstein distance is considered to define a rigorous distance between images, applied to non-rigid image registration and warping. 
As a distance function for shapes the Gromov-Wasserstein distance is introduced in a series of works by Memoli \cite{memoli2007use,
memoli2011gromov}. Here, shapes are modelled as compact metric spaces and the Gromov-Wasserstein distance is computed on isometry classes of each space. The Gromov-Wasserstein differs from the Wasserstein distance \eqref{pwasserdist} as it assigns a cost to pairs of transport assignments. 
In the context of surface dissimilarity measurement  
the Wasserstein distance is applied to metric densities on the hyperbolic disc representing conformal mappings of different surfaces \cite{lipman2011conformal}. 
In \cite{schmitzer2013modelling} the authors create a convex shape-prior from a modified Gromov-Wasserstein distance. Their approach can be used for image segmentation problems in which prior shape knowledge on the objects that should be segmented can be provided in terms of a template shape. 
This approach is modified in \cite{schmitzer2013object} using the
quadratic Wasserstein distance as a regularizer between learned reference shapes and the segmentation.
In \cite{burger2012regularized} the Wasserstein distance is used as a data fidelity term in a generic regularization approach and applied for image density estimation and cartoon-texture decomposition. 
Moreover, a review of the use of geodesic methods and in particular optimal transportation in computer vision can be found in \cite{peyre2010geodesic}.

{\bf Flow of diffeomorphism.} The physical modeling of viscous flow involves dissipation as an integrated measure of local friction. Arnold \cite{Ar66a,ArKh98} proposes to study viscous flows from the perspective of a family of diffeomorphisms $(\phi(t))_{t\in [0,1]}: \bar \domain \to \domain$ which describe the transport of densities, \eg image intensities, along particle paths  $(\phi(t,x))_{t\in [0,1]}$ for $x\in \domain$. This concept is picked up in vision by Grenander and coworkers \cite{grenander1981lectures, dupuis1998variational}. As a Riemannian metric $\G^{\mathrm{diff}}[\cdot, \cdot]$ one considers the rate of viscous dissipation induced by the Eulerian flow velocity 
$v(t) = \dot \phi(t) \circ \phi^{-1}(t)$ in a multipolar fluid model (cf. Ne\v{c}as and \v{S}ilhav\'y \cite{NeSi91}). 
Here, the Eulerian motion field $v$ is considered as a tangent vector on the manifold of diffeomorphisms.  
The resulting Riemannian metric one obtains $\tilde \G^{\mathrm{diff}}[v,v] := \int_\domain L[v,v] \diff x$, where $L[v,v] = C \varepsilon[v] : \varepsilon[v]$ with $\varepsilon[v]=(\nabla v + (\nabla v)^T)/2$,
and a deduced path energy
$\tilde \E^{\mathrm{diff}}[\phi] = \int_0^1 \tilde \G^{\mathrm{diff}}[v,v] \diff t$
as an action functional on flows encodes the total accumulated dissipation on the domain $\domain$ and on the time interval $[0,1]$.
To study the warping of two image intensity functions 
$\u_A,\u_B: \domain \to \R$ for which there exists a diffeomorphism $\phi$ with $\u_B = \u_A\circ \phi$ 
a flow minimizing the energy $\E^{\mathrm{diff}}$ subject to the constraints $\phi(0) = \Id$ and $\u_B = \u_A\circ \phi(1)$ defines a geodesic path 
$(\u(t))_{t\in [0,1]}$ with 
$\u(t) = \u_A \circ \phi^{-1}(t)$  in the spaces of images connecting $\u_A$ and $\u_B$.  
If we aim at deriving a Riemannian distance directly between images via the flow of diffeomorphism approach, then a motion field $v$ can be viewed as a representation of an image variation. Obviously, different motion fields might represent the same image variation. Hence, the corresponding equivalence class $\overline v$ is considered as a tangent vector on the image manifold and the associate metric is now given by 
\beqn\label{eq:fdmetric}
\G^{\mathrm{diff}}[\overline v, \overline v] :=   \min_{v \in \overline{v}} \int_\domain L[v,v] \diff x \, .
\eeqn
Consequently, the path energy on a path $(\u(t))_{t\in [0,1]}$ reads as 
\beqn\label{eq:fdenergy}
\E^{\mathrm{diff}}[\u] = \int_0^1 \G^{\mathrm{diff}}[\overline{v},\overline{v}] \diff t\,,
\eeqn
which one minimizes over all image paths with $\u(0) = \u_A$, $\u(1) = \u_B$. Here, we assume that there is at least one path of finite path energy connecting 
$\u_A$ and $\u_B$. In medical applications \cite{BeMiTrYo02} each diffeomorphisms represents a particular anatomic configuration of an anatomic reference structures. For more details we refer to \cite{DuGrMi98,BeMiTr05,JoMi00,MiTrYo02}.

{\bf Metamorphosis.}
A one-to-one correspondence of image grey values in warping applications is frequently not realistic.  The metamorphosis approach offers a suitable generalization of  the flow of diffeomorphism concept. It is first presented by Miller and Younes \cite{MiYo01}. A rigorous analytical treatment is due to Trouv\'e and Younes \cite{TrYo05}. In addition to the transport of image intensities along motion paths the variation of an intensity value along a motion path is allowed and reflected by an additional term in the energy. This term measures the integrated squared material derivative $z=\partial_t \u + \nabla \u \cdot v$. 
From a geometric perspective a pair of material derivative $z$ and motion velocity $v$ represent a variation of an image $\u$. 
Hence, an equivalence class $\overline{(z,v)}$ of all pairs which generate the same image variation is considered as a tangent vector 
on the image manifold. A Riemannian metric $\G^{\mathrm{meta}}[\overline{(z,v)},\overline{(z,v)}]$ acts on these tangent vectors 
and the associated path energy along an image path $(\u(t))_{t\in [0,1]}$ is given by 
\beqn\label{eq:menergy}
\E^{\mathrm{meta}}[\u] = \int_0^1 \G^{\mathrm{meta}}[\overline{(z,v)},\overline{(z,v)}] \diff t\,.
\eeqn
As an example for the underlying Riemannian metric we obtain 
\beqn\label{eq:mmetric}
\G^{\mathrm{meta}}[\overline{(z,v)},\overline{(z,v)}] = \min_{(z,v) \in \overline{(z,v)}} \int_\domain  L[v,v]  + \frac1\penaltyPushforward z^2 \diff x\,,
\eeqn
where the first three terms in the integrant retrieve the metric from the flow of diffeomorphism approach and encode the induced viscous dissipation, whereas the last term penalizes temporal changes of intensities along motion paths.

In this paper, we combine the optimal transportation approach with the metamorphosis approach. Thereby, in addition to the transportation cost we take into account a density variation of the transported measure and viscous dissipation. The paper is organized as follows:
First we present our generalized image transport model in Section \ref{sec:generalized}.
In Section \ref{sec:existenceNonViscous} we prove existence of geodesics in the non-viscous case.
Then we propose a variational time discretization of the full model in Section \ref{sec:timediscrete}, prove existence of time discrete geodesics in Section \ref{sec:discreteexistence} and describe a fully discrete solution scheme in Section \ref{sec:SpatialDiscretization}.
Furthermore, we consider in Section \ref{sec:BB} the algorithm used in \cite{BeBr00} to compute for comparison reasons geodesics in the purely non-viscous case.
Finally, we generalize in Section \ref{sec:barycentre} our model to discrete weighted barycentres and apply it to textures and objects.

\section{The generalized image transport model}\label{sec:generalized}
In this section we will discuss the generalization of the optimal transport model in image warping and blending. 
These generalizations are motivated by two observations in applications:\\
-- Frequently, objects or structures in images, which are in correspondence and are expected to be matched via the transport, have different masses. From a global perspective the assumptions that images are considered as probability distributions is too restrictive. Indeed, the latter requires in advance contrast modulation, which is somewhat artificial. In the classical optimal transport model local mass differences lead to artifacts, where a local mass surplus has to be deposited elsewhere without any structural correspondence. We will no longer enforce the source free transport equation and explicitly incorporate a source term in the path energy which measures density modulation.\\
-- Different from the flow of diffeomorphism approach the optimal transport maps are not necessarily homeomorphisms. On the other hand in many applications one is interested in topological consistency and at the same time the physical background of the application might suggest to incorporate a dissipative term in the path energy. Hence, we combine the classical transport cost model with a weighted viscous dissipation model.

To this end we first recall the formulation \eqref{eq:BB} of the Wasserstein distance proposed by Benamou and Brenier \cite{BeBr00}.
In what follows we restrict to a bounded domain $\domain \subset \R^d$ ($d=2,3$) with Lipschitz boundary.
Now, we allow for a source term $z:[0,1] \times \domain \to \R$ in the transport equation  defined for given image intensity 
$\u:[0,1] \times \domain \to \R$ and transport field $v:[0,1] \times \domain \to \R^d$ as
\beqn\label{eq:source}
z := \partial_t \u + \div (\u v)\,.
\eeqn
Furthermore, we pick up the model for the viscous dissipation in \eqref{eq:fdenergy} and obtain as a new path energy
\beqn\label{eq:energy}
\Edg[\u] = \int_0^1 \Gdg[\overline{(z,v)},\overline{(z,v)}] \diff t\,.
\eeqn
with 
\beqn\label{eq:action}
\Gdg[\overline{(z,v)},\overline{(z,v)}] = \min_{(z,v) \in \overline{(z,v)}} \int_\domain  \u |v|^2 + \frac1\penaltyPushforward z^2 + \gamma L[v,v]   \diff x\,,
\eeqn
which we minimize subject to \eqref{eq:source} and the constraints
$\u(0) = \u_A$ and $\u(1) = \u_B$. Here, $\u_A$ and $\u_B$ are the given input images and the $\overline{(z,v)}$ is the equivalence class of 
pairs of a source term and a transport field, which are consistent with the transport equation \eqref{eq:source} for given image intensity $\u$.
The involved local rate of viscous dissipation is given by  
$L[v,v] =  \tfrac{\lambda}{2} (\tr\varepsilon[v])^2+ \mu\tr(\varepsilon[v]^2) + \penaltyHigherOrderDerivative |D^m v|^2$, 
where $\varepsilon[v]= \frac12(\nabla v +  \nabla v^T)$, $m>1+\frac{d}{2}$ and $\lambda, \, \mu,\, \eta  >0$. 
(The first two terms represent the viscous dissipation of a Newtonian fluid and the higher order terms reflect a multipolar viscosity). 
As in \cite{TrYo05} and similar to \cite{ambrosio2006gradient} the condition $\partial_t \u + \div(\u v) = z$ has to be understood in weak form 
\begin{align*}
 \int_0^1 \int_\domain \eta z \diff x \diff t = - \int_0^1 \int_\domain (\partial_t \eta + v \cdot \nabla \eta) \u \diff x \diff t
\end{align*}
for all $\eta \in C_c^{\infty}( (0,1) \times \domain)$.

The first term in the metric is the classical transport cost rate, the second terms reflects the source term which measures the density modulation of the image intensity and the last term 
is the dissipation rate based on a multipolar viscous fluid model. 
Let us emphasize that for general non divergence free motion fields 
$z$ does not coincide with the material derivative as in the metamorphosis model \cite{TrYo05}. We suppose that $\gamma \geq 0$ measuring the impact of viscosity and $\penaltyPushforward >0$
is a penalty parameter weighting the impact of density modulation on the metric and the path energy. In the formal limit $\penaltyPushforward \to \infty$ and for $\gamma=0$ we retrieve 
the standard transport cost. Given the path energy, we can define a Riemannian (generalized Wasserstein) 
distance $\Wdg[\u_A,\u_B]$ of two images $\u_A$ and $\u_B$ as
\beqn\label{eq:newtransport}
\Wdg[\u_A,\u_B]^2 = \min_{{(\u(t))_{t\in [0,1]}}\atop {\u(0) = \u_A,\; \u(1) = \u_B}} \Edg[\u]\,.
\eeqn

\section{Existence of geodesics for the non-viscous model ($\gamma=0$)}\label{sec:existenceNonViscous}

In this section we study existence of minimizers $\u$ of \eqref{eq:energy} in the non-viscous case, that is for $\gamma=0$. In order to give a rigorous proof, it will be necessary to reformulate the formal problem \eqref{eq:energy} as a problem for measures rather than for densities. In particular, it will be crucial to treat the singular parts of the measures in an appropriate way.

Following \cite{BeBr00}, it will be useful to replace the velocity variable $\bfv$ by the momentum variable $\bfw = \u \bfv$. Therefore, the function $\u |v|^2$ appearing in the path energy \eqref{eq:action} will be replaced by $|\bfw|^2 / \u$. The joint convexity of this function will play a crucial role in the sequel.

The argument presented here is a modification of the argument in \cite{DNS09} and our presentation follows this latter work very closely. Some additional arguments are needed to deal with the possibly varying total mass. On the other hand, some simplifications can be made, since we work on a bounded spatial  domain instead of the whole space $\R^d$.

First we reformulate the action functional \eqref{eq:action}. Let $\Omega$ be a bounded domain in Euclidean space and fix a reference measure $\Leb \in \OMeas$. In our application, the domain $\Omega$ will either be the spatial domain $\domain$ or the space-time domain $[0,1] \times \domain$, and $\Leb$ will be the corresponding Lebesgue measure. 

Let $\mu \in \OMeas$, $\bfnu \in \OSMeasd$, and $\zeta \in \OSMeas$. The Lebesgue decomposition of these measures with respect to $\Leb$ is given by
\begin{align*}
 \mu = \u \Leb + \mu^\perp\;, \qquad
 \bfnu = \bfw \Leb + \bfnu^\perp\;, \qquad
 \zeta = z \Leb + \zeta^\perp\;.
\end{align*}
Let now $\Leb^\perp \in \OMeas$ be such that $\mu^\perp, \bfnu^\perp$ and  $\zeta^\perp$ are absolutely continuous with respect to $\Leb^\perp$ (take for instance $\Leb^\perp = \mu^\perp + |\bfnu^\perp| + |\zeta^\perp|$). Then we may write  
\begin{align*}
  \mu^\perp = \u^\perp \Leb^\perp\;, \qquad
 \bfnu^\perp = \bfw^\perp \Leb^\perp \;, \qquad
 \zeta^\perp = z^\perp \Leb^\perp\;.
\end{align*}
As in the Benamou-Brenier formulation of the $2$-Wasserstein distance we consider the function $\phi : [0,\infty) \times \R^d \to [0,\infty]$ defined by
\begin{align*}
 \phi(\u,\bfw)  =  \left\{ \begin{array}{ll}
 0  
  & \u = 0 \text{ and } \bfw = 0\;,\\
 \frac{|\bfw|^2}{\u}
  & \u > 0\;,\\
 + \infty
  & \u = 0 \text{ and } \bfw \neq 0\;
.\end{array} \right.
\end{align*}
Note that $\phi$ is lower-semicontinuous, convex and 1-homogeneous.
The action functional $\cA : \OMeas \times \OSMeasd\times \OSMeas \to [0,+\infty]$ that we are interested in is given by
\begin{align*}
 \cA(\mu, \bfnu, \zeta) := \ABB(\mu, \bfnu) + \frac{1}{\penaltyPushforward}\AM(\zeta) \;, \qquad
 \ABB(\mu, \bfnu) 
 &  :=  \int_\Omega \phi(\u, \bfw) \dd \Leb
     +\int_\Omega \phi(\u^\perp, \bfw^\perp) \dd \Leb^\perp\;,\\
 \AM(\zeta) 
 &  := 
 \left\{ \begin{array}{ll}
 \int_\Omega z^2 \dd \Leb
 & \zeta^\perp = 0\;,\\
+ \infty & \zeta^\perp \neq 0\;.\end{array} \right. 
\end{align*}

Since $\phi$ is jointly $1$-homogeneous, the definition of $\ABB$ does not depend on the choice of $\Leb^\perp$. The same is true for $\AM$, since we may write $\AM(\zeta) = \int_\Omega z^2 \dd \Leb +  \int_\Omega \psi(z^\perp) \dd \Leb^\perp$, where $\psi : \R \to [0,+\infty]$ is the $1$-homogeneous function defined by $\psi(0) = 0$ and $\psi(r) = + \infty$ for $r \neq 0$.
Sometimes it will be useful to write $\cA^\Omega$ instead of $\cA$ in order to emphasize the domain $\Omega$.

The following result is an immediate consequence of general lower-semicontinuity results for integral functionals on measures \cite{AB88,AFP00}.

\begin{Proposition}[Lower semicontinuity of the functional $\cA$]\label{prop:A-lsc}
Consider weak$^*$-convergent sequences of measures 
\begin{align*}
 \mu_n \weakly^* \mu \in \OMeas\;, \qquad
 \bfnu_n \weakly^* \bfnu \in \OSMeasd\;, \qquad
 \zeta_n \weakly^* \zeta \in \OSMeas\;.
\end{align*}
Then we have
$\cA(\mu, \bfnu, \zeta)  \leq \liminf_{n \to \infty}
\cA(\mu_n, \bfnu_n, \zeta_n)\,$.
\end{Proposition}

\begin{proof}
The result follows, since both $\ABB$ and $\AM$ satisfy the assumptions of \cite[Theorem 2.1]{DNS09}.
\end{proof}

The following crucial lemma is a special case of \cite[Proposition 3.6]{DNS09}.

\begin{Lemma}[Integrability estimate]\label{lem:O-lemma}
Let $\mu \in \OMeas$ and $\bfnu \in \OSMeasd$.
For any Borel function $\eta : \Omega \to \R_+$ we have
\begin{align*}
 \int_\Omega \eta(x) \dd |\bfnu|(x) 
 \leq \Big(\ABB(\mu, \bfnu)\Big)^{\frac12} \bigg(\int_\Omega \eta^2 \dd \mu\bigg)^{\frac12}\;.
\end{align*}
\end{Lemma}

\begin{proof}
Set $S := \{ x \in \Omega : \eta(x) > 0\}$.
Using the scalar inequality $\sqrt{ab} + \sqrt{cd} \leq \sqrt{a+c}\sqrt{b+d}$ which holds for $a,b,c,d \geq 0$, we obtain
\begin{align*}
 \int_S \eta(x) \dd |\bfnu|(x) 
 &   =  \int_S \eta |\bfw| \dd \Leb  
        +  \int_S \eta |\bfw^\perp| \dd \Leb^\perp 
 \\& \leq \bigg( \int_S \phi(\u, \bfw) \dd \Leb       \bigg)^{\frac12}
 \bigg( \int_S \eta^2 \u \dd \Leb   \bigg)^{\frac12}
 + 
 \bigg( \int_S \phi(\u^\perp, \bfw^\perp)  \dd \Leb^\perp   \bigg)^{\frac12}
 \bigg( \int_S \eta^2 \u^\perp \dd \Leb^\perp  \bigg)^{\frac12}
\\& \leq \bigg( \int_S \phi(\u, \bfw) \dd \Leb   
    +  \int_S \phi(\u^\perp, \bfw^\perp) \dd \Leb^\perp  \bigg)^{\frac12}
 \bigg( \int_S \eta^2 \u \dd \Leb   + 
  \int_S \eta^2 \u^\perp \dd \Leb^\perp  \bigg)^{\frac12}
\\& \leq \Big(\ABB(\mu, \bfnu)\Big)^{\frac12} \bigg(\int_S \eta^2 \dd \mu\bigg)^{\frac12}\;.
\end{align*}
\end{proof}

Let us now introduce the modified continuity equation.

\begin{Definition}[A continuity equation without conservation of mass]\label{def:cont-eq}
Let $t \mapsto \mu_t$ be weak$^*$-continuous in $\Meas$, let $t \mapsto \bfnu_t$ be Borel measurable in $\SMeasd$, and let $t \mapsto \zeta_t$ be Borel measurable in $\SMeas$.
We say that the triple $(\mu_t, \bfnu_t, \zeta_t)_{t \in [0,1]}$ satisfies the continuity equation (and write $(\mu_t, \bfnu_t, \zeta_t)_{t \in[0,1]} \in \CE[0,1]$) if 
\begin{enumerate}
\item the following integrability conditions hold:
\begin{align*}
 \int_0^1 |\bfnu_t|(\domain) \dd t < \infty\;, \qquad
 \int_0^1 |\zeta_t|(\domain) \dd t < \infty\;; 
\end{align*}
\item the modified continuity equation $\partial_t \mu + \div(\bfnu) = \zeta$ holds in the sense of distributions, \ie, for all space-time test functions $\eta \in C^1_0( (0,1) \times \overline{\domain})$ we have
\begin{align*}\label{eq:CE-dist}
  \int_0^1 \int_\domain \partial_t \eta(t,x) \dd \mu_t(x) \dd t 
    + \int_0^1 \int_\domain \nabla \eta(t,x) \cdot \dd \bfnu_t(x) \dd t 
    + \int_0^1 \int_\domain \eta(t,x) \dd \zeta_t(x) \dd t  
    = 0\;.
\end{align*}
\end{enumerate}
\end{Definition}

A standard approximation argument (see \cite[Lemma 4.1]{DNS09}) shows that solutions to $\CE[0,1]$ satisfy, for all $0 \leq t_0 \leq t_1 \leq 1$,
\begin{equation}\begin{aligned}\label{eq:CE-2}
     \int_\domain \eta(t_1,x) d \mu_{t_1}(x)
  -  \int_\domain \eta(t_0,x) d \mu_{t_0}(x) 
  & = \int_{t_0}^{t_1} \int_\domain \partial_t \eta(t,x) \dd \mu_t(x) \dd t  + \int_{t_0}^{t_1} \int_\domain \nabla \eta(t,x) \cdot \dd \bfnu_t(x) \dd t 
 \\& \qquad + \int_{t_0}^{t_1}  \int_\domain \eta(t,x) \dd \zeta_t(x) \dd t
\end{aligned}\end{equation}
for all space-time test functions  
$\eta \in C^1( [0,1] \times \overline{\domain})$. In particular, taking $\eta(t,x) \equiv 1$, it follows that the increase of mass is given by
\begin{align}\label{eq:mass-increase}
 \mu_{t_1}(\domain) - \mu_{t_0}(\domain) 
  = \int_{t_0}^{t_1} \zeta_t(\domain) \dd t\;.
\end{align}

We are now in a position to rigorously define the extended distance $\Wd$ that was formally introduced in \eqref{eq:newtransport}.

\begin{Definition}\label{def:rigorous}
For $\mu_A, \mu_B \in \Meas$ we define $\Wd(\mu_A,\mu_B) \in [0,+\infty]$ by
\begin{align}\label{eq:def-Wd}
 \Wd(\mu_A, \mu_B)
    := \inf_{\mu,\bfnu,\zeta} \bigg\{ \bigg(\int_0^1  \cA(\mu_t, \bfnu_t, \zeta_t)      \dd t \bigg)^{1/2} \ : \ (\mu_t, \bfnu_t, \zeta_t)_{t \in[0,1]} \in \CE[0,1]\;,  \ \mu_0 = \mu_A\;, \ \mu_1 = \mu_B  \bigg\}\;.
\end{align}
\end{Definition}

The following theorem is the main result of this section.
\begin{Theorem}[Existence of geodesics]\label{thm:geod}
Let $\penaltyPushforward \in (0,\infty)$ and 
take $\mu_A, \mu_B \in \Meas$ with $\Wd(\mu_A, \mu_B) < \infty$. Then there exists a minimizer $(\overline\mu_t, \overline\bfnu_t, \overline\zeta_t)_{t \in[0,1]}$ that realizes the infimum in \eqref{eq:def-Wd}. Moreover, the associated curve $(\overline\mu_t)_{t \in[0,1]}$ is a constant speed geodesic for $\Wd$, \ie, 
\begin{align*}
 \Wd(\overline\mu_s, \overline\mu_t) = |s-t| 
 \Wd(\mu_A, \mu_B) 
\end{align*}
for all $s,t \in [0,1]$. Furthermore, we have the alternative characterisation
\begin{align*}
  \Wd(\mu_A, \mu_B)
    := \inf_{\mu,\bfnu,\zeta} \bigg\{ \int_0^1  \sqrt{\cA(\mu_t, \bfnu_t, \zeta_t)  }    \dd t \ : \ (\mu_t, \bfnu_t, \zeta_t)_{t \in[0,1]} \in \CE[0,1]\;,  \ \mu_0 = \mu_A\;, \ \mu_1 = \mu_B  \bigg\}\;.
\end{align*}
\end{Theorem}

\begin{proof}
The existence of a minimizer is an immediate consequence of Proposition \ref{prop:compact} below. The remaining statements follow by standard arguments, see \cite[Theorem 5.4]{DNS09} for details.
\end{proof}

Let us now state and prove the main ingredient for the proof of Theorem \ref{thm:geod}. We write $\int_0^1  \delta_t \otimes \mu_t \dd t$ to denote the measure $\mu$ on $[0,1] \times D$ satisfying
\begin{align*}
 \int_{[0,1] \times \domain} \eta(t,x) \dd \mu(t,x) = 
\int_0^1  \int_{\domain}  \eta(t,x) \dd \mu_t(x) \dd t
\end{align*}
for all $\eta \in C([0,1] \times \domain)$.

\begin{Proposition}[Compactness for solutions to the continuity equation with bounded action]\label{prop:compact}
Suppose that\\ $(\mu_t^n, \bfnu_t^n, \zeta_t^n)_{t \in(0,1)} \in \CE[0,1]$ satisfy 
\begin{enumerate}
\item[(A1)] $M_1 := \sup_n \mu_0^n(\domain) < \infty$;
\item[(A2)] $M_2 := \sup_n  \int_0^1 \cA(\mu_t^n, \bfnu_t^n, \zeta_t^n) \dd t < \infty$.
\end{enumerate}
Set $\bfnu^n := \int_0^1 \delta_t \otimes \bfnu_t^n \dd t \in \STMeasd$ and $\zeta^n := \int_0^1  \delta_t \otimes \zeta_t^n \dd t \in \STMeas$.
Then, there exists a subsequence (again indexed by n) and a triple $(\mu_t, \bfnu_t, \zeta_t)_{t \in(0,1)} \in \CE[0,1]$ such that 
\begin{enumerate}
\item $\mu_t^n \weakly^* \mu_t$ in $\Meas$ for all $t \in [0,1]$;
\item $\bfnu^n \weakly^* \bfnu$ in $\STMeasd$;
\item $\zeta^n \weakly^* \zeta$ in $\STMeas$.
\end{enumerate}
Moreover, for the above subsequence
\begin{align}\label{eq:lsc}
 \int_0^1 \cA(\mu_t, \bfnu_t, \zeta_t) \dd t
   \leq \liminf_{n \to \infty}
    \int_0^1 \cA(\mu_t^n, \bfnu_t^n, \zeta_t^n) \dd t\;.
\end{align}
\end{Proposition}

\begin{proof}
In view of (A2), we first observe that
\begin{align*}
 M_3 := \sup_n \int_{0}^{1} |\zeta_t^n|(D) \dd t  
   = \sup_n \int_{0}^{1} \int_\domain |z_t^n(x)| \dd x \dd t 
   < \infty\;.
\end{align*}
Therefore, \eqref{eq:mass-increase} yields the uniform bound
\begin{align*}
 \mu_t^n(\domain) \leq 
  \mu_0^n(\domain) +  
 \int_{0}^{t} |\zeta_s^n|(\domain) \dd s \leq M_1 + M_3
\end{align*}
for all $n$.
Moreover, Lemma \ref{lem:O-lemma} implies that 
\begin{align*}
 |\bfnu_t^n|(\domain) \leq \sqrt{\mu_t^n(\domain)\ABB(\mu_t^n,\bfnu_t^n) }\;,
\end{align*}
hence by the H\"older inequality we obtain 
\begin{align*}
\int_0^1 |\bfnu_t^n|(\domain)^2 \dd t 
   \leq \int_0^1 \mu_t^n(\domain) \cA(\mu_t^n, \bfnu_t^n, \zeta_t^n) \dd t
   \leq M_2 (M_1 + M_3)\;,
\end{align*}
which shows that the maps $\{t \mapsto |\bfnu_t^n|(\domain)\}_n$ are uniformly bounded in $L^2(0,1)$, hence uniformly integrable.

Since $|\bfnu^n|([0,1] \times \domain) \leq (\int_0^1 |\bfnu_t^n|(\domain)^2 \dd t)^{\frac12}$, the measures $\{\bfnu^n\}_n \in \STMeasd$ have uniformly bounded total variation on $[0,1] \times \domain$, hence we can extract a subsequence that converges weakly$^*$ to some measure $\bfnu \in \STMeasd$. The uniform integrability of $\{t \mapsto |\bfnu_t^n|(\domain)\}_n$ implies that the image measure of $\bfnu$ under the mapping $(t,x) \mapsto t$ is absolutely continuous with respect to the Lebesgue measure on $[0,1]$. Therefore, the disintegration theorem (see, \eg, \cite[Theorem 5.3.1]{ambrosio2006gradient}) allows us to write $\bfnu = \int_0^1 \delta_t \otimes \bfnu_t \dd t$ for some family of measures $\{\bfnu_t\}_{t \in [0,1]} \in \SMeasd$.

Fix $0 \leq \tau \leq 1$, take $\eta \in C^1(\domain)$, and set $\bar \xi(t,x) := \nabla\eta(x)\chi_{[0,\tau]}(t)$. 
Although $\bar\xi$ is discontinuous, general approximation results (see \cite[Proposition 5.1.10]{ambrosio2006gradient}) imply that 
\begin{align}\label{eq:nu-conv}
 \int_{0}^{\tau} \int_\domain \nabla \eta(x) \dd \bfnu_t^n(x) \dd t
  =   \int_{[0,1] \times \domain} \bar \xi(t,x) \dd \bfnu^n(t,x)
  \to \int_{[0,1] \times \domain} \bar \xi(t,x) \dd \bfnu(t,x)
  = \int_{0}^{\tau} \int_\domain \nabla \eta(x) \dd \bfnu_t(x) \dd t\;.
\end{align}

Let us now consider the term involving $\zeta_t^n$, which is treated similarly. 
For all $n$ and a.e. $t \in [0,1]$ we use (A2) to conclude that  $\zeta_t^n = z_t^n \Leb$. Therefore we obtain
\begin{align*}
 \int_0^1 |\zeta_t^n|(\domain)^2 \dd t 
   = \int_0^1 \bigg( \int_\domain |z_t^n(x)| \dd x \bigg)^2 \dd t 
   < \infty\;.
\end{align*}
As above, we infer that the mappings $\{t \mapsto |\zeta_t^n|(\domain)\}_n$ are uniformly integrable, and that there exists a subsequence of $\{\zeta^n\}_n$ that convergence weakly$^*$ to some measure $\zeta\in \STMeas$. By the disintegration theorem we may write $\zeta = \int_0^1 \delta_t \otimes \zeta_t \dd t$ for a family of measures $\{\zeta_t\}_{t \in [0,1]} \in \SMeas$.
Set $\tilde\xi(t,x) := \eta(x) \chi_{[0,\tau]}(t)$. Arguing as above, we obtain
\begin{align}\label{eq:zeta-conv}
  \int_{0}^{\tau} \int_\domain \eta(x) \dd \zeta_t^n(x) \dd t
  =   \int_{[0,1] \times \domain} \tilde\xi(t,x)\dd \zeta^n(t,x)
  \to \int_{[0,1] \times \domain} \tilde\xi(t,x) \dd \zeta(t,x)
  = \int_{0}^{\tau} \int_\domain \eta(x) \dd \zeta_t(x) \dd t\;.
\end{align}

We are now in a position to obtain subsequential convergence of $\{\mu_t^n\}_n$. Indeed, it follows from \eqref{eq:CE-2} that
\begin{align*}
      \int_\domain \eta(x) d \mu_{\tau}^n(x)
  =  \int_\domain \eta(x) d \mu_{0}^n(x) 
   +  \int_0^{\tau} \int_\domain \nabla \eta(x) \cdot \dd \bfnu_t^n(x) \dd t + \int_{0}^{\tau}  \int_\domain \eta(x) \dd \zeta_t^n(x) \dd t\;.
\end{align*}
Moreover, (A1) implies that there exists a measure $\mu_0 \in \Meas$ such that $\mu_0^n \weakly^* \mu_0$ (after passing to a subsequence). In view of \eqref{eq:nu-conv} and \eqref{eq:zeta-conv} the latter equation implies weak$^*$-convergence of $\{\mu_\tau^n\}_n$ to some measure $\mu_\tau$ for every $\tau \in [0,1]$. It is readily checked that $(\mu_t, \bfnu_t, \zeta_t)_t \in \CE[0,1]$.

It remains to prove \eqref{eq:lsc}.
For this purpose we write $\mu = \int_0^1 \delta_t \otimes \mu_t \dd t$ and $\mu^n = \int_0^1  \delta_t \otimes \mu_t^n \dd t$. 
It is straightforward to check that $\mu^n$ converges weakly$^*$ to $\mu$ in $\STMeasp$. Now the result follows by observing that 
\begin{align*}
 \int_0^1 \cA^\domain(\mu_t, \bfnu_t, \zeta_t) \dd t
  = \cA^{[0,1] \times \domain}(\mu, \bfnu, \zeta)\;,
\end{align*}
and applying Proposition \ref{prop:A-lsc} to $\cA^{[0,1]\times \domain}$.
\end{proof}

\section{A variational time discretization}\label{sec:timediscrete}
In what follows, we derive a time discrete approximation of the energy \eqref{eq:energy} and thereby a variational approach for the definition 
of geodesic paths. We refer to \cite{WiBaRu10} for the general concept and to \cite{RuWi12b} for the numerical analysis in the context of shape spaces which are Hilbert manifolds and 
in \cite{BeEf14} a variational time discretization of geodesics in the metamorphosis model is discussed.

As a motivation let us briefly present a toy model in finite dimensions. 
On a smooth $m$-dimensional manifold $\mathcal{M}$ embedded in $\R^d$ ($m\leq d$) we consider the simple energy $\F[y, \tilde y] = |\tilde y - y|^2$
which reflects the stored elastic energy in a spring spanned between points $y$ and $\tilde y$ through the ambient space of $\mathcal{M}$ in $\R^d$.
The smoothness of $\mathcal{M}$ implies that $\F[y, \tilde y]  = \mbox{dist}_{\mathcal{M}}(y,\tilde y)^2 + O( \mbox{dist}_{\mathcal{M}}(y,\tilde y)^3)$, where 
$\mbox{dist}_{\mathcal{M}}(y,\tilde y)$ denotes the Riemannian distance between $y$ and $\tilde y$.
Hence, we can approximate the path length of a smooth path $(y(t))_{t\in [0,1]}$ via sampling $y_k = y(\frac{k}{K})$ and then evaluate the discrete path energy
$$
\EK[y_0,\ldots, y_K] = K \sum_{k=1}^K |y_{k}-y_{k-1}|^2\,,
$$
such that $\EK[y_0,\ldots, y_K]$ converges to $\E[(y(t))_{t\in [0,1]}]  = \int_0^1 |\dot y(t)|^2 \diff t$ for $K\to \infty$. Here, we use that $\frac{y_k-y_{k-1}}{\tau}$
is an approximation of the velocity $\dot y(k/K)$ where $\tau = \frac1K$ is the time step size of our discretization on the time interval $[0,1]$.
In fact, based on the approximation $\F$ of the squared distance 
$\mbox{dist}_{\mathcal{M}}$, which is easy to implement, we obtain an effective approximation of the Riemannian path energy 
$\E$. Correspondingly, we call a minimizer $(y_0,\ldots, y_K)$ of the discrete path energy for fixed $y_0$ and $y_K$ a discrete geodesic.
We refer to \cite{RuWi12b,BeEf14} for a detailed discussion, why the discrete path energy instead of the discrete path length is the right concept to compute discrete geodesics. In particular, $\Gamma$-convergence of the discrete path energy is proven in case of the metamorphosis model in \cite{BeEf14} and
under suitable assumptions in the context of Hilbert manifolds in \cite{RuWi12b}.

Now, we ask for a similar time discrete approximation of the continuous path energy $\Edg$ defined in \eqref{eq:energy}.
To this end, we consider a discrete path $(\u_0, \ldots, \u_K)$ in the space of image intensities with $\u_k\in \Imagespace$
 for $k=1,\ldots, K$ with $\Imagespace  := L^2(\domain,\R_{\geq 0})$
and ask for a matching functional $\F$ on consecutive pairs $\u$, $\tilde \u$ of image intensities. In fact, this matching functional should reflect time discrete counterparts of 
all three ingredients of the metric $\Gdg$ and the induced continuous path energy $\Edg$, namely the transport cost, the viscous dissipation and the source term.
Like in the original Monge problem we take into account deformations $\phi$ in a suitable space $\deformationSpace$ of admissible deformations, 
to be defined later, and optimize for given $\u$, $\tilde \u$
a suitable functional  $\Fd[\u, \tilde \u, \cdot]$ over all admissible deformations
to define the value of the matching functional $\F[\u, \tilde \u]$, \ie 
\begin{align*}
 \F[\u, \tilde \u] = \inf_{ \phi \in \mathcal{A} }  \Fd[\u, \tilde \u, \phi] \,.
\end{align*}
With the matching functional at hand we then define the discrete path energy summing over applications of the matching functional to consecutive pairs of image intensities $\u_{k-1}$ and $\u_k$ 
of a discrete path $(\u_0, \ldots, \u_K)$ and get
\beqn\label{eq:discreteenergy}
\EdgK[\u_0,\ldots, \u_K] = K \sum_{k=1}^K \F[\u_{k-1}, \u_k]\,.
\eeqn
Thus, the resulting time discrete approximation of the squared Riemannian distance is given by
\beqn\label{eq:WdgK}
\WdgK[\u_A,\u_B]^2 = \min_{\substack{ \u_0, \ldots, \u_K \in \Imagespace \\ \u_0 = \u_A,\; \u_K = \u_B}} \EdgK[\u_0,\ldots, \u_K]\,.
\eeqn
Here, we assume $\u_A, \u_B \in \Imagespace$.
In what follows we list now the appropriate components of $\F$ reflecting the different ingredients of the continuous path energy.

{\bf Approximation of the transport cost.} 
To approximate the first term in the metric \eqref{eq:action} we make use of the equivalence of the original 
Monge problem and the Benamou Brenier formulation \cite{BeBr00} of optimal transport and define  
\beqn\label{eq:discreteTC}
 \Fdt[\u, \phi] = \int_\domain  | \phi - \id |^2  \u \diff x \,.
\eeqn
Here, $\frac{\phi-\id}{\tau}$ is an approximation of the transport velocity with $\Id$ being the identity deformation.

{\bf Approximation of the density modulation cost.}
For a diffeomorphism $\phi$ the push forward condition $\phi_{\#} (\u \Leb) =  \tilde \u \Leb$ can be expressed as 
\begin{align*}
 \u = \det( D \phi )  \tilde \u \circ \phi \,.
\end{align*}
As an approximation of the source term $z=\partial_t \u + \div (v \u)$ we take into account 
\beqn
 \Fdz[\u, \tilde \u, \phi] = \int_\domain \frac{1}{\penaltyPushforward} | \det( D \phi ) \tilde \u \circ \phi - \u |^2 \diff x \,.
 \nonumber
 \label{eq:EnergyWDZ}
\eeqn
{\bf Approximation of the dissipation cost.}
By Rayleigh's paradigm~\cite{St45} one
derives models for viscous dissipation from elastic energies replacing
elastic strains by strain rates. We proceed as in \cite{BeEf14}, where a time discretization of the metamorphosis model was investigated 
and define 
\begin{align*}
\Fdv[\phi] = \gamma \int_\domain   W(D\phi) + \epsilon |D^m \phi|^2 \diff x \,.
\label{eq:EnergyWDV}
\end{align*}
Here, $W$ is a hyper elastic energy density and 
the higher order term $|D^m \phi|^2$ acts as a regularizing term for some small $\epsilon >0$ and enforces the deformations to be in the space $W^{m,2}$. We make the following assumptions on $W$ (cf. also \cite{BeEf14}):
\setcounter{counter}{0}
\begin{list}{(W\arabic{counter})}{\usecounter{counter}}
  \item $W$ is non-negative and polyconvex,
  \item $W(A)\geq \alpha_0 | \log \det A | - \alpha_1$ for $\alpha_0,\alpha_1>0$, and every invertible matrix $A$ with $\det A >0$, $W(A) = \infty$ for $\det A \leq 0$, and 
  \item $W$ is sufficiently smooth and the following consistency assumptions with respect to the differential operator $L$ hold true:
  $W(\Id) = 0$, $DW(\Id) = 0$ and  $\frac12 D^2 W(\Id)(B,B)=\frac{\lambda}{2}(\tr B)^2+\mu \tr((\frac{B+B^T}{2})^2)$  for all $B \in \R^{d,d}$.
\end{list}
Due to the incorporation of this dissipation energy we finally define the space of admissible deformations over which we minimize in the definition of $\F[\u, \tilde \u]$ as 
\begin{align*} 
 \deformationSpace = \left\{ \phi \in W^{m,2}(\domain,\domain) \, : \, \det( D \phi ) > 0 \; \text{ a.e. in } \domain, \phi = \Id \; \text{ on } \partial D \right\} \, ,
\end{align*}
We assume that $m > 1 + \frac{d}{2}$, which implies by Sobolev embedding that the admissible deformations are diffeomorphisms.
Given these energy contributions we can define the compound energy 
\beqn\label{eq:F}
\Fd[\u, \tilde \u, \phi] = \Fdt[\u, \phi] + \Fdz[\u, \tilde \u, \phi] + \Fdv[ \phi]\,.
 \nonumber
\eeqn
The following interpolation results justifies our choice of the time discrete path energy.
\begin{Theorem}[Consistency of the discrete path energy]
For a convex domain $\domain$ and a sufficiently smooth path of image intensities $(\u(t))_{t\in [0,1]}$
with $\u\geq 0$ a.e. in $[0,1]\times \domain$
and a sufficiently smooth family of  velocities $(v(t))_{t\in [0,1]}$ 
we consider interpolated images $\u^K_k = \u(\tfrac{k}{K})$ and motion fields $v^K_k=v(\tfrac{k}{K})$.
Then the resulting extended path energy
\beqn\label{EdgKd}
\EdgKd[\u^K_0,\ldots, \u^K_K, \phi^K_1,\ldots, \phi^K_K] := \sum_{k=1}^K \Fd[\u^K_{k-1}, \u^K_k,\phi^K_k]
\eeqn
with $\phi^K_k = \tfrac{1}{K} v^K_k + \id$ converges to the corresponding continuous path energy
\beqn\label{Edgd}
\Edgd[\u,v] := \int_0^1 \int_\domain  \u |v|^2  + \frac1\penaltyPushforward z^2 + \gamma L[v,v] \diff x \diff t\,.
\nonumber
\eeqn
\end{Theorem}
\begin{proof}
We define the step size $\tau = \frac{1}{K}$.
First, for the transport cost we easily get
\begin{align*}
 K \sum_{k=1}^K \int_\domain |\phi_k^K - \id|^2 \u_{k-1}^K \diff x
 = \sum_{k=1}^K \tau \int_\domain |v_k^K|^2 \u_{k-1}^K \diff x
 \xrightarrow{K \to \infty} \int_0^1 \int_\domain |v|^2 \u \diff x \diff t \, .
\end{align*}
Following \cite{WiBaRu10} the convergence of the dissipation cost follows from a Taylor expansion of the hyperelastic density function $\Wreg$ by using the consistency assumptions:
\begin{align*}
  &K \sum_{k=1}^K \int_\domain \Wreg(D\phi_k^K) + \penaltyHigherOrderDerivative |D^m \phi_k^K|^2 \diff x \\
  =& \frac{1}{\tau} \sum_{k=1}^K \int_\domain \Wreg(\Id) + \tau DW(\Id)(Dv_k^K) + \frac{\tau^2}{2} D^2W(\Id)(Dv_k^K,Dv_k^K) + O(\tau^3) + \penaltyHigherOrderDerivative \tau^2 |D^m v_k^K|^2 \diff x \\
  =& \sum_{k=1}^K \tau \int_\domain \left( \frac{\lambda}{2} \left(\tr Dv_k^K\right)^2 + \mu \tr\left( \frac{Dv_k^K+(Dv_k^K)^T}{2}\right)^2 \right)  + O(\tau) + \penaltyHigherOrderDerivative |D^m v_k^K|^2 \diff x
  \xrightarrow{K \to \infty} \int_0^1 \int_\domain L[v,v] \diff x \diff t \, .
\end{align*}
Finally, for the density modulation cost we use the Taylor expansions
\begin{align*}
 \det D\phi_k^K &= \Id + \tau \tr \left( \frac{ D \phi_k^K - \id}{\tau} \right) + O(\tau^2) = \Id + \tau \div( v_k^K ) + O(\tau^2) \\
 \u_k^K \circ \phi_k^K &= \u_k^K + \tau \nabla \u_k^K \cdot v_k^K + O(\tau^2)
\end{align*}
and obtain
\begin{align*}
 & K \sum_{k=1}^K \int_\domain |\det(D \phi_k^K ) \u_k^K \circ \phi_k^K - \u_{k-1}^K|^2 \diff x \\
 =& \sum_{k=1}^K \tau \int_\domain \left| \frac{\u_k^K - \u_{k-1}^K}{\tau} + \div( v_k^K ) \u_k^K + \nabla \u_k^K \cdot v_k^K + O(\tau) \right|^2 \diff x 
 \xrightarrow{K \to \infty} \int_0^1 \int_\domain |\partial_t \u + \div(\u v )|^2 \diff x \diff t \,.
\end{align*}

\end{proof}
\section{Existence of time discrete geodesics}\label{sec:discreteexistence}
In this section we assume that the assumptions of Section \ref{sec:timediscrete} are fulfilled
and that $\penaltyPushforward,\gamma >0$. As before we assume that $m > 1 + \frac{d}{2}$.
We will show that for given images 
$\u_A, \u_B \in \Imagespace = L^2(\domain,\R_{\geq 0})$ 
a time discrete geodesic exists.
First we prove that $\F$ is well-posed in the sense that there is an optimal deformation between two images.
\begin{Proposition}[Existence of minimizing deformations]\label{wellPosednessWEnerImL2}
Let $\u, \tilde \u \in \Imagespace$. 
Then $\Fd[\u, \tilde \u, \phi]$ attains its minimum over all deformation $\phi \in \deformationSpace$.
Moreover, $\phi$ is a diffeomorphism and $\phi^{-1} \in C^{1,\alpha}(D)$ for $\alpha \in (0, m - 1 - \frac{d}{2} )$.
\end{Proposition}
\begin{proof}
  {\em Step 1}. 
First, we observe that $\Fd$ is bounded from below, since $\u$ is non-negative by definition of $\Imagespace$, $\W$ is non-negative by assumption (W1) and the source term is non-negative anyway.
Because of (W2)  and $\phi = \id \in \deformationSpace$ there exists an upper bound for the energy $\Fd$ on a minimizing sequence $(\phi^j)_{j \in \N}$.
Following \cite{BeEf14} one observes that a subsequence, again denoted by $(\phi^j)$, converges weakly in $W^{m,2}(\domain,\domain)$ to some $\phi \in W^{m,2}(\domain,\domain)$
and for the limit deformation we get $\phi^{-1} \in C^{1,\alpha}(\domain)$.

{\em Step 2}.
We prove that $\phi \mapsto \Fd[\u, \tilde \u, \phi]$ is lower semicontinuous w.r.t. weak convergence in $L^2$.
It is sufficient to show that $\det(D\phi^j) \tilde \u \circ \phi^j \rightharpoonup \det(D\phi) \tilde \u \circ \phi$ in $L^2$.
Then the result follows from the weak lower semicontinuity of the $L^2$-norm, the compact embedding of $W^{m,2}(\domain,\domain)$ into $C^{1,\alpha}(\domain,\domain)$ for $0<\alpha < m-\frac{d}{2}$, and results on the weak lower semicontinuity 
of polyconvex functionals \cite{Ci88}.
 By the assumption $\tilde \u \in L^2(\domain)$ and by Step 1 
 we have a uniform $L^2$-bound on $\det(D\phi^j) \tilde \u \circ \phi^j$, 
 so it is enough to prove that the expression converges in the sense of distributions.
  For $\eta \in C_c^{\infty}(\domain)$ we have
  \begin{align*}
    \int_\domain \det(D\phi^j)(x) \tilde \u \circ \phi^j(x) \eta(x) \diff x &= \int_\domain \tilde \u (x) \eta \circ (\phi^j)^{-1}(x) \diff x \quad \text{and} \\
    \int_\domain \det(D\phi)(x) \tilde \u \circ \phi(x) \eta(x) \diff x &= \int_\domain \tilde \u (x) \eta \circ \phi^{-1}(x) \diff x\,.
  \end{align*}
  Since $(\phi^j)^{-1} \rightarrow \phi^{-1}$ in $C^{1,\alpha}$, the result follows by the dominated convergence theorem.
\end{proof}
Now, for a given discrete path $(\u_0, \ldots, \u_K) \in \Imagespace^{K+1}$ 
we consider $\EdgKd$ defined in \eqref{EdgKd}.
By Proposition \ref{wellPosednessWEnerImL2} there exists $\Phi  = (\phi_1, \ldots, \phi_K) \in \deformationSpace^K$ such that $\EdgKd[(\u_0, \ldots, \u_K), (\phi_1, \ldots, \phi_K)] = \EdgK[(\u_0, \ldots, \u_K)]$. Now we study $\EdgKd$ for a fixed vector of deformations.
\begin{Proposition}\label{existenceImageInnerVecL2}
 Let $\u_A, \u_B \in \Imagespace$, $K \geq 2$. 
 Then for a fixed vector of deformations $\Phi = (\phi_1, \ldots, \phi_K) \in \mathcal{A}^K$ there exists a unique discrete path $(\u_0, \ldots , \u_{K}) \in \Imagespace^{K+1}$ with $\u_0 =\u_A$ and $\u_K = \u_B$, \ie
 \begin{align*}
  \EdgKd[(\u_0, \ldots , \u_{K}), \Phi] = \inf\limits_{(\u_1, \ldots , \u_{K-1}) \in \Imagespace^{K-1}} \EdgKd[(\u_A,\u_1,\ldots,\u_{K-1},\u_B), \Phi] \, .
 \end{align*}
\end{Proposition}
\begin{proof}
First we see that the functional is bounded from above on a minimizing sequence by computing the energy of $( \u_B, \ldots, \u_B ) \in \Imagespace^{K-1}$:
 \begin{align*}
       \EdgKd[(\u_A, \u_B, \ldots, \u_B, \u_B), (\phi_1, \ldots, \phi_K) ] 
  \leq C(\phi_1, \ldots, \phi_k ) ( 1 + \left\| \u_A \right\|_{2} + \left\| \u_A \right\|_{2}^2 + \left\| \u_B \right\|_{2} + \left\| \u_B \right\|_{2}^2 ) < \infty \,.
 \end{align*}
Next we observe that for fixed $(\phi_1, \ldots, \phi_K) \in \deformationSpace^K$ the time discrete path energy is quadratically growing, \ie
$$
\EdgKd[(\u_A, \u_1 \ldots , \u_{K-1}, \u_{B}), (\phi_1, \ldots, \phi_K)] \geq c_1 \sum_{i=1,\ldots,K-1}  \|\u_k\|^2_{L^2(\domain)}  - c_2
$$ 
for constants $c_1,c_2 >0$ depending on $(\phi_k)_k$.
Therefore we can take a minimizing sequence $(\u_1^j, \ldots, \u_{K-1}^j)_{j \in \N} \subset \Imagespace^{K-1}$, which has because of the upper bound a weakly converging subsequence in $L^2$ with limit $(\u_1, \ldots, \u_{K-1}) \in \Imagespace^{K-1}$.
Now, the energy is strictly convex in $\u_k$ for all $k=1,\ldots, K-1$, hence there is a unique minimizer in $\Imagespace^{K-1}$.
\end{proof}

Next, we can use these two propositions to prove existence of minimizers of the discrete path energy $\EdgK$.
\begin{Theorem}[Existence of discrete geodesics]
 Let $\u_A, \u_B \in \Imagespace$, $K \geq 2$ be given.
 Then there exists $(\u_1,\ldots,\u_{K-1}) \in \Imagespace^{K-1}$ s.t.
 \begin{align*}
  \EdgK[(\u_A, \u_1, \ldots, \u_{K-1}, \u_B) ] = \inf\limits_{(\tilde{\u}_1, \ldots, \tilde{\u}_{K-1}) \in \Imagespace^{K-1} } \EdgK[ (\u_A, \tilde{\u}_1, \ldots, \tilde{\u}_{K-1}, \u_B ) ] \,.
 \end{align*}
\end{Theorem}
\begin{proof}
Taking $\u_k^j = \u_B$ and $\phi_k = \Id$ to test the energy $\EdgK$ we observe that the $\EdgK$  is bounded from above on a minimizing sequence $(\u_1^j,\ldots,\u_{K-1}^j)_{j \in \N}$.
Take a minimizing sequence $(\u_1^j,\ldots,\u_{K-1}^j)_{j \in \N}$ of the discrete path energy $\EdgK[(\u_A, \cdot , \u_B)]$.
Due to Proposition \ref{wellPosednessWEnerImL2}, for every $(\u_1^j,\ldots,\u_{K-1}^j)$ there exists a family of optimal deformations $(\phi_1^j, \ldots, \phi_K^j) \in \mathcal{A}^K$ with 
$\EdgKd[(\u_A, \u_1^j,\ldots,\u_{K-1}^j, \u_B), (\phi_1^j, \ldots, \phi_K^j)] \leq \EdgKd[(\u_A, \u_1^j,\ldots,\u_{K-1}^j, \u_B), \Psi]$ for all $\Psi \in \mathcal{A}^K$.
As in the proof of Proposition \ref{wellPosednessWEnerImL2} there exists a subsequence again denoted $(\phi_k^j)_{j\in \N}$ with 
$\phi_k^j \rightharpoonup \phi_k$ in $W^{m,2}$ for all $k=1,\ldots, K$, s.t. $\phi_k^{-1} \in C^{1,\alpha}$.
By Proposition \ref{existenceImageInnerVecL2} we can assume (possible replacing $(\u_1^j,\ldots,\u_{K-1}^j)$ and thereby further reducing the energy) that $(\u_1^j,\ldots,\u_{K-1}^j)$ already minimizes the  energy $\EdgKd[(\u_A,\cdot,\u_B),(\phi_1^j,\ldots,\phi_{K}^j)]$ in $\Imagespace^{K-1}$.
Then $\u_k^j$ is uniformly bounded in $L^2$ by a constant $C$ depending only on $\u_A$ and $\u_B$ for $k=1,\ldots, K$. 
This constant $C$ is independent of the $\phi_k^j$ due to 
the uniform bound of $\phi_k^j$ in $W^{m,2}$. 
Hence we can pass to a further subsequence satisfying $(\u_1^j,\ldots,\u_{K-1}^j) \rightharpoonup (\u_1,\ldots,\u_{K-1})$ in $L^2$.
To prove weak lower semicontinuity in $L^2$ of the functional, it is sufficient to pass to the limit in the identities
\begin{align}
    \int_\domain \det(D\phi_k^j)(x) \u_k^j \circ \phi_k^j(x) \eta(x) \diff x &= \int_\domain \u_k^j (x) \eta \circ (\phi_k^j)^{-1}(x) \diff x\,, \\
    \int_\domain \det(D\phi_k)(x) \u_k \circ \phi_k(x) \eta(x) \diff x &= \int_\domain \u_k (x) \eta \circ \phi_k^{-1}(x) \diff x\,,
\end{align}
for an arbitrary $C^\infty$-function $\eta$, which follows from the $C^{1,\alpha}$-convergence of $(\phi_k^j)^{-1}$ and the weak $L^2$-convergence of $\u_k^j$.    For the demonstration of lower semicontinuity in the remaining terms we refer to analogous discussion in Proposition \ref{wellPosednessWEnerImL2}.
\end{proof}

Finally, let us study in more detail the optimality conditions for $(\u_1, \ldots, \u_{K-1}) \in  \Imagespace^{K-1}$ in preparation of the later derivation of a numerical algorithm. At first we consider the simplified model without the constraint $\u_k \geq 0$ for $k=1,\ldots,K-1$. Since for fixed deformations the energy is strictly convex, there exists a unique minimizer. For each $k=1,\ldots,K-1$ there are two terms in the energy where $\u_k$ appears:
 \begin{align*}
   \Fd[\u_{k}, \u_{k+1}, \phi_{k+1}]
  = & \int_\domain | \phi_{k+1} - \id |^2  \u_{k} + \frac{1}{\penaltyPushforward} | \det( D \phi_{k+1} ) \u_{k+1} \circ \phi_{k+1} - \u_k |^2 \diff x + \gamma \Fdv[\phi_{k+1}]\,, \\
    \Fd[\u_{k-1}, \u_k, \phi_k]
  = & \int_\domain | \phi_k - \id |^2  \u_{k-1} + \frac{1}{\penaltyPushforward} | \det( D \phi_k ) \u_k \circ \phi_k - \u_{k-1} |^2 \diff x + \gamma \Fdv[\phi_k] \\
  = & \int_\domain | \phi_k - \id |^2  \u_{k-1} + \frac{1}{\penaltyPushforward}  | \u_k - ( \det(D \phi_k)^{-1} \u_{k-1} ) \circ \phi_k^{-1} |^2 \det(D \phi_k) \circ \phi_k^{-1} \diff x + \gamma \Fdv[\phi_k] \,.
 \end{align*}
 Hence, the Euler-Lagrange equation for $\u_k$ is 
 \begin{align*}
  0 = | \phi_{k+1} - \id |^2 - \frac{2}{\penaltyPushforward} ( \det( D \phi_{k+1} ) \u_{k+1} \circ \phi_{k+1} - \u_k ) 
  + \frac{2}{\penaltyPushforward} ( \u_k - ( \det(D\phi_k)^{-1} \u_{k-1} ) \circ \phi_k^{-1} ) \det(D\phi_k) \circ \phi_k^{-1}
 \end{align*}
 for all $k=1,\ldots,K-1$ and a.e. $x \in \domain$.
 Now we define the discrete transport path $X(x) = \left( X_1(x), X_2(x), \ldots, X_{K-1}(x) \right)$ with $X_1(x) = \phi_1(x)$ and $X_k(x) = \phi_k (X_{k-1} (x) )$
 and the vector 
 \begin{align*}
  \bar \u(x) = \left( \u_1(X_1(x)), \u_2(X_2(x)), \ldots , \u_{K-1}(X_{K-1}(x)) \right) \,.
 \end{align*}
 Then we can write the optimality conditions as 
 \begin{align*}
  \u_k = \frac{  \det( D \phi_{k+1} ) \u_{k+1} \circ \phi_{k+1} +  \u_{k-1} \circ \phi_k^{-1}  - \frac{\penaltyPushforward}{2} | \phi_{k+1} - \id |^2  }{ 1 + \det(D\phi_k) \circ \phi_k^{-1} } \, .
 \end{align*}
 From $X_k \in C^{1,\alpha}$ we deduce that
 \begin{align}
  \u_k \circ X_k = \frac{  (\det( D \phi_{k+1} ) \circ X_k) (\u_{k+1} \circ X_{k+1}) +  \u_{k-1} \circ X_{k-1}  - \frac{\penaltyPushforward}{2} | X_{k+1} - X_k |^2  }{ 1 + \det(D\phi_k) \circ X_{k-1} }
  \label{recursionFormulaWithoutDet}
 \end{align}
 for a.e. $x\in \domain$ and for all $k=1,\ldots, K-1$.
 This can be rewritten as a linear system $\MatrixBF(x) \bar \u(x) = \RBF(x)$,
 where $\MatrixBF(x)$ is a tridiagonal matrix given by
 \begin{align*}
  \MatrixBF(x)_{k,k+1} & = - \frac{ \det(D \phi_{k+1} ) \circ X_k(x) }{ 1 + \det(D \phi_k) \circ X_{k-1}(x) } , 
  \quad
  \MatrixBF(x)_{k,k} = 1,
  \quad
  \MatrixBF(x)_{k,k-1} = - \frac{ 1 }{ 1 + \det(D \phi_k) \circ X_{k-1}(x) }
 \end{align*}
 and $\RBF(x) = \SBF(x) + \TBF(x)$ with $\SBF(x), \TBF(x) \in \R^{K-1}$ given by
 \begin{align*}
  \SBF(x) & = \left( \frac{\u_A(x)}{ 1 + \det(D \phi_1)(x) }, 0, \ldots , 0, \frac{ \det( D \phi_K) \circ X_{K-1}(x) \u_B \circ X_K(x) }{ 1 + \det( D \phi_{K-1}) \circ X_{K-2}(x) } \right)^T \\
  \TBF(x)_k & = - \frac{\penaltyPushforward}{2} \frac{|X_{k+1}(x) - X_k(x)|^2}{1+\det(D \phi_k) \circ X_{k-1}(x)} \quad \forall k=1,\ldots,K-1\,.
 \end{align*}
Now, the unique minimizer $(\u_1, \ldots, \u_{K-1}) \subset \Imagespace^{K-1}$ satisfies for a.e. $x \in \domain$ the derived linear system of equations and gives the only solution of this system. Thus $A(x)$ is invertible for a.e. $x \in \domain$ and by solving the system we can recover the minimizer.
In the constraint case $\u \geq 0$ a.e. the minimization with respect to $(\u_1,\ldots, \u_{K-1})$ no longer decomposes into a linear system 
of equations with unknowns $(\u_1(X_1(x)), \ldots, \u_{K-1}(X_{K-1}(x)))$ for a.e. $x \in \domain$. But the decomposition along the discrete paths $(X_0(x),\ldots, X_K(x))$ is still applicable. Indeed, one observes that for a.e. $x \in \domain$ the  vector $(\u_1(X_1(x)), \ldots, \u_{K-1}(X_{K-1}(x)))$
minimizes the quadratic functional 
\beq
Q(\tilde \u_1, \ldots, \tilde \u_{K-1}) 
= \sum_{k=1}^{K} \det(D X_{k-1})(x) \left(| X_{k}(x) - X_{k-1}(x) |^2  \tilde \u_{k-1} + 
\frac{1}{\penaltyPushforward} | \det( D \phi_{k})(X_{k-1}(x)) \tilde \u_{k}  - \tilde \u_{k-1} |^2 \right) 
\eeq
with $\tilde \u_0 = \u_A(x)$ and $\tilde \u_K = \u_B(x)$ over all $(\tilde \u_1, \ldots, \tilde \u_{K-1}) \in \R^{K-1}$  subject to the constraint $\tilde \u_k \geq 0$ for all $k=1,\ldots, K-1$. This is a simple quadratic optimization problem in $\R^{K-1}$ with inequality constraints.

\section{Spatial discretization}\label{sec:SpatialDiscretization}
With respect to the spatial discretization we follow the procedure already proposed in \cite{BeEf14}. 
We restrict to two dimensional images ($d=2$) and 
consider a regular quadrilateral grid on the two-dimensional image domain $\domain=[0,1]^2$ consisting of rectangular cells $\left\{ C_m \right\}_{m \in I_C }$ with $I_C$ being the associated index set. Let $\V_h$ be the space of piecewise bilinear continuous functions and denote by
$\left\{ \basisfct^i \right\}_{i\in I_N}$ the set of nodal basis functions with $I_N$ being the index set of all grid nodes $x_i$. \\
We investigate spatially discrete deformations $\Phi_k: \domain\to \domain$ with $\Phi_k \in \V_h^2 = \V_h\times \V_h$ and spatially discrete image maps $\U_k: \domain \to \R$ with $\U_k \in \V_h$.
Given any finite element function $U \in \V_h$ we denote by $\bar U = (U(x_i))_{i\in I_N}$ the corresponding vector of nodal values.
Now, we define a fully discrete counterpart $\EdgKh$ of the so far solely time discrete path energy $\EdgK$ defined in \eqref{eq:discreteenergy} as follows
\begin{align*}
 & \EdgKh[(\U_0,\ldots, \U_K)]  =  \min_{\substack{\Phi_k \in \V_h^2 \\ \Phi_k|_{\partial \domain}= \Id}}
   \EdgKdh[(\U_0,\ldots, \U_K),(\Phi_1,\ldots, \Phi_K)] 
\end{align*}
and obtain the resulting fully discrete approximation of the squared Riemannian distance 
\beqn\label{eq:WdgKh}
\WdgKh[\U_A,\U_B]^2 = \min_{{\U_0, \ldots, \U_K \in \Imagespace} \atop {\U_0 = \U_A,\; \U_K = \U_B}} \EdgKh[\U_0,\ldots, \U_K]\,.
\nonumber
\eeqn
Here, $\EdgKdh[(\U_0,\ldots, \U_K),(\Phi_1,\ldots, \Phi_K)]$ is the discrete counterpart of $\EdgKd$ in \eqref{EdgKd} 
obtained by the evaluation of all the integrals in $\EdgKd$ using third order Simpson quadrature with $9$ quadrature points.
Then, the resulting entries of the weighted mass matrix $\mass_h[\omega,\Phi,\Psi]_{ij} = \left(\mass_h[\omega,\Phi,\Psi]_{ij} \right)_{i,j \in I_N}$ with weight $\omega$ and transformed via deformations $\Phi, \Psi$ are given by 
$$
\mass_h[ \omega,\Phi,\Psi]_{ij} = \sum_{l\in I_C}\sum_{q=0}^8 w_q^l \, \omega (x_q^l)\,   (\basisfct^i\circ \Phi)(x_q^l)\,(\basisfct^j \circ \Psi)(x_q^l)\,.
$$
Here, the $x_q^l$ are the quadrature points and the $w_q^l$ are corresponding quadrature weights. 
In the case $\omega = 1$ we write $\mass_h[1,\Phi,\Psi] = \mass_h[\Phi,\Psi]$. 

To compute a minimizer of the fully discrete energy $\EdgKh$ we proceed as in the existence proof of time discrete geodesics in Section \ref{sec:discreteexistence} and alternate the optimisation of the set of deformations for fixed image intensities and the optimization of the image intensities for fixed deformations.
The optimization of deformations decouples in time. 
To calculate an optimal, discrete matching deformation for two consecutive images we use a conjugate gradient method for the fully discrete energy $\EdgKdh$.
In practice  we use the following hyperelastic energy
$W(D\phi) = \frac{\mu}{2} \left\| D\phi \right\|_F^2 + \frac{\lambda}{4} \left( \det D\phi \right)^2 - (\mu + \frac{\lambda}{2}) \log( \det D\phi ) - \mu - \frac{\lambda}{4}$ for $\det(D\phi) > 0$
with fixed $\lambda  = 10$ and $\mu = 1$
and differing from the assumptions in Section \ref{sec:timediscrete} we skip the higher order term $|D^m \phi|^2$. Indeed, the associated regularization experimentally turned out not to be necessary, possibly due to the regularization by the spatial discretization.
For a fixed vector of discrete deformations ${\mathbf{\Phi}} =  (\Phi_1,\ldots, \Phi_{K})$ the minimization of $\EdgKdh$ with respect to $ \Uinnervec = (\U_1, \ldots,  \U_{K-1})$ leads, as in the spatially continuous case, to a linear system of equations. 
Indeed, we obtain as the discrete counterpart of $\int_\domain |\phi_k - \id|^2 \u_{k-1} \diff x$ 
\begin{align*}
     \sum_{k=1}^K \sum_{l\in I_C} \sum_{q=0}^8 w_q^l  \left( | \Phi_k - \id |^2  \U_{k-1} \right) (x_q^l) 
  =  \sum_{k=1}^K \mass_h[ | \Phi_k - \id|^2, \Id, \Id] \bar \U_{k-1} \bar 1 \, ,
\end{align*}
with $\bar 1 = (1,\ldots, 1) \in \R^{I_N}$ and as the discrete counterpart of  $\int_\domain |\det(D\phi_k) \u_k \circ \phi_k - \u_{k-1}|^2 \diff x$ 
\beqa
   && \sum_{k=1}^K \sum_{l\in I_C} \sum_{q=0}^8 w_q^l \left( | \det (D\Phi_k) (\U_{k} \circ \Phi_k) - \U_{k-1}|^2 \right) (x_q^l) \\
 && =  \sum_{k=1}^K \left(\mass_h[ (\det D \Phi_k)^2, \Phi_k,\Phi_k] \bar \U_k \cdot \bar \U_k - 2 \mass_h[ \det( D \Phi_k), \Phi_k, \Id] \bar \U_k \cdot \bar \U_{k-1} + \mass_h[\Id,\Id] \bar \U_{k-1} \cdot \bar \U_{k-1} \right) \,.
\eeqa
Hence, the resulting discretized part of $\EdgKdh$ depending on $\bar \U_k$ is given by
\beqa                    
&& \mass_h[ | \Phi_{k+1} - \id|^2, \Id, \Id] \bar \U_{k} \bar 1 
 + \frac{1}{\penaltyPushforward}  (\mass_h[ (\det D \Phi_k)^2, \Phi_k,\Phi_k] + \mass_h[\Id,\Id]) \bar \U_k \cdot \bar \U_k \\
&& - \frac{2}{\penaltyPushforward}  \left( \mass_h[ \det( D \Phi_k), \Phi_k, \Id] \bar \U_k \cdot \bar \U_{k-1} + \mass_h[ \det( D \Phi_{k+1}), \Phi_{k+1}, \Id] \bar \U_{k+1} \cdot \bar \U_k \right)\,.
\eeqa
In what follows, we restrict to the non constraint case minimizing over intensities, which are not necessarily non negative. In fact, in our numerical experiments for $\U_A, \U_B \geq 0$ and with all deformations being initialized with the identity we did not observe negative density values in the vectors $\bar \U_k$ for $k=1,\ldots, K-1$. The implementation of a constraint, quadratic optimization method is work in progress.

For the variation of $\EdgKdh$ with respect to the $k$-th image $\bar \U_k$ one obtains
\beqa   
    \partial_{\bar \U_k} \EdgKdh &= & \mass_h[ | \Phi_{k+1} - \id|^2, \Id, \Id] \bar 1
  + \frac{2}{\penaltyPushforward} (\mass_h[ (\det D \Phi_k)^2, \Phi_k,\Phi_k] + \mass_h[\Id,\Id] ) \bar \U_k \\ 
 && -  \frac{2}{\penaltyPushforward} \left( \mass_h[ \det D \Phi_k, \Phi_k, \Id]^T \bar \U_{k-1} + \mass_h[ \det D \Phi_{k+1}, \Phi_{k+1}, \Id] \bar \U_{k+1} \right) \,. 
\eeqa   
As a consequence the necessary condition for $\Uinnervec := (\U_1,\ldots,  \U_{K-1})$ to be a minimizer of $\EdgKdh$
is a block tridiagonal system of linear equations $\MatrixBF[\mathbf{\Phi}] \bar \Uinnervec = \RBF[\mathbf{\Phi}]$, where
$\MatrixBF[\mathbf{\Phi}]$ is formed by $(K-1)\times(K-1)$ matrix blocks $\MatrixBF_{k,k'} \in  \R^{I_N\times I_N}$ with
\begin{align*}
\MatrixBF_{k,k-1} & = - \mass_h[ \det D \Phi_k, \Phi_k, \Id]^T, \quad 
\MatrixBF_{k,k}    = \mass_h[ (\det D \Phi_k)^2, \Phi_k,\Phi_k] + \mass_h[\Id,\Id]\,,\\ 
\MatrixBF_{k,k+1} & = - \mass_h[ \det D \Phi_{k+1}, \Phi_{k+1}, \Id]
\end{align*}
and $\RBF[\mathbf{\Phi}] = \SBF[\mathbf{\Phi}] + \TBF[\mathbf{\Phi}]$ consists of $K-1$ vector blocks $\RBF_k = \SBF_k + \TBF_k \in \R^{I_N}$ with
$\SBF_1 = \mass_h[ \det D \Phi_1, \Phi_1, \Id]^T \bar \U_0$, $\SBF_2 = \ldots = \SBF_{K-2}     = 0$, $\SBF_{K-1}  = \mass_h[ \det D \Phi_{K}, \Phi_{K}, \Id] \bar \U_K$, and 
$\TBF_k  = - \frac{\penaltyPushforward}{2} \mass_h[ | \Phi_{k+1} - \id|^2, \Id, \Id]^T \bar 1$ for all $k=1,\ldots, K-1$.

The energy $\sum_{l\in I_C}\sum_{q=0}^8 w_q^l \left( |\Phi_k - \id|^2 \U_{k-1} + \frac{1}{\penaltyPushforward} | \det D\Phi_k \U_{k} \circ \Phi_k- \U_{k-1}|^2 \right)(x_q^l)$
is convex in $\U_{k}$ and strictly convex in $\U_{k-1}$.
Hence, $\EdgKdh$ is strictly convex in $\Uinnervec$ and there is a unique minimizer $\Uinnervec = \Uinnervec[\mathbf{\Phi}]$ for fixed $\mathbf{\Phi}$.
This implies that $\MatrixBF$ is invertible and therefore the resulting solution $\Uinnervec$ coincides with the unique minimizer of $\EdgKdh$.
Numerically, the corresponding system of linear equations is solved with a conjugate gradient method with diagonal preconditioning.
In addition, as an outer iteration of the numerical energy descent scheme we apply a cascadic approach starting on coarse grids and successively refining the grid.   
 
In what follows, we will discuss numerical results obtained by the proposed scheme. 
We start with two simple transport examples of image densities with identical mass. 
\begin{figure}[t]
\setlength{\unitlength}{.053\linewidth}
\resizebox{0.85\linewidth}{!}{
  \begin{picture}(23.2,2)(0,0)
        \put(0, 0){\includegraphics[width=0.1\textwidth]{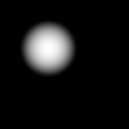}}
    \put(2.2, 0){\includegraphics[width=0.1\textwidth]{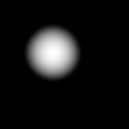}}
    \put(4.4, 0){\includegraphics[width=0.1\textwidth]{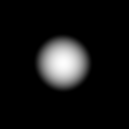}}
    \put(6.6, 0){\includegraphics[width=0.1\textwidth]{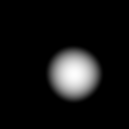}}
    \put(8.8, 0){{\includegraphics[width=0.1\textwidth]{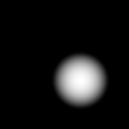}}}
            \put(12.0, 0){\includegraphics[width=0.1\textwidth]{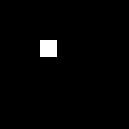}}
    \put(14.2, 0){\includegraphics[width=0.1\textwidth]{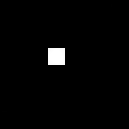}}
    \put(16.4, 0){\includegraphics[width=0.1\textwidth]{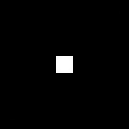}}
    \put(18.6, 0){\includegraphics[width=0.1\textwidth]{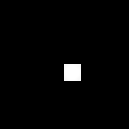}}
    \put(20.8, 0){{\includegraphics[width=0.1\textwidth]{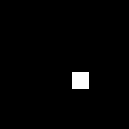}}}
  \end{picture}
}
\caption{Optimal transport via translation for two different pairs of images, each of them with identical mass. Left: discrete geodesic between a pair of scaled bump maps with $K=4$, $\penaltyPushforward = 10^{-1}$, $\gamma=10^{-1}$ is shown, right: discrete geodesic between a square and a translated square with $K=4$, $\penaltyPushforward = 10^{-1}$, $\gamma=10^{-2}$.}
\label{fig:SimpleExamples}
\end{figure}
In Figure \ref{fig:SimpleExamples} the optimal transport geodesics connecting a 
bump map $f(x) = \exp\left(\left(1-\sigma^{-2} |x-x_0|^2\right)^{-1}\right) \chi_{B_\sigma(x_0)}$ with centre $x_0\in \domain$ and radius $\sigma >0$
and its translate as well as a characteristic function of a square and its translate are considered for small $\penaltyPushforward$ and $\gamma$. 
Indeed, the computed optimal transport constitutes of a translation.
Next, we illustrate the role of the source term allowing for density modulation in case of $\u_A$ and $\u_B$ given in Figure \ref{fig:BenamouBrenierExampleDiffMassCompare} as two bump maps of different size at different centre points and in Figure \ref{fig:differentMass} as the characteristic functions of two 
rectangles of different size still for small $\gamma$.
\begin{figure}[h]
\setlength{\unitlength}{0.05\textwidth}
\resizebox{0.85\linewidth}{!}{
  \begin{picture}(19.7,8)
        \put(0,4){\includegraphics[width=0.18\textwidth]{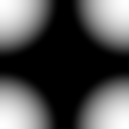}}
    \put(4,4){\includegraphics[width=0.18\textwidth]{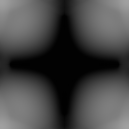}}
    \put(8,4){\includegraphics[width=0.18\textwidth]{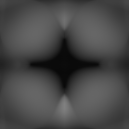}}
    \put(12,4){\includegraphics[width=0.18\textwidth]{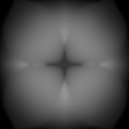}}
    \put(16,4){{\includegraphics[width=0.18\textwidth]{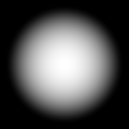}}}
        \put(0,0){\includegraphics[width=0.18\textwidth]{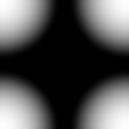}}
    \put(4,0){\includegraphics[width=0.18\textwidth]{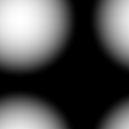}}
    \put(8,0){\includegraphics[width=0.18\textwidth]{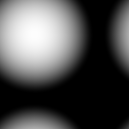}}
    \put(12,0){\includegraphics[width=0.18\textwidth]{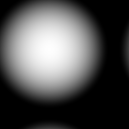}}
    \put(16,0){{\includegraphics[width=0.18\textwidth]{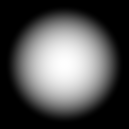}}}
  \end{picture}
}
\caption{Discrete geodesic are computed for different viscosity parameter $\gamma$ between two bump maps placed on a square and periodically extended to $\R^2$.  Top: $\gamma = 5 \cdot 10^{-4}$, $\penaltyPushforward = 10^{-1}$,  bottom: $\gamma = 5$ ($\penaltyPushforward = 10^{-1}$). The images are extracted from a discrete geodesic with $K=9$. 
The different contributions to the 
resulting discrete path energy are: transport cost = 0.0278363, density modulation cost = 0.00107356, dissipation cost = 0.0128489 (top row)  and  transport cost = 0.155598, density modulation cost = 0.00274782, dissipation cost =  0.00139578 (bottom row).  }
\label{fig:differentViscosity}
\end{figure}
In Figure \ref{fig:differentViscosity} we show the influence of the viscous dissipation. 
Picking up a test case from \cite{BeBr00} we consider image intensities on a square periodically extended to $\R^2$ with a bump map once placed in the vertices of the square and once at the centre. We consider periodic boundary conditions both for the image intensities and for the motion field. The Wasserstein geodesic was already computed in \cite{BeBr00} and we obtain approximately the same result for small $\gamma$. Indeed, the bump map at the vertices split up into four pieces, which are then transported separately into the centre. From the perspective of optimal transport this path is energetically preferable due to the shorter transport distance compared to a simple translation from the vertices into the centre. Obviously, this splitting of mass is expensive from the viscous dissipation perspective. Hence, for larger $\gamma$ we observe the simple translation.

Next, we illustrate the role of the source term allowing for density modulation in case of $\u_A$ and $\u_B$ as in Figure \ref{fig:BenamouBrenierExampleDiffMassCompare} but now with different mass in the two bump maps.
Still we impose periodic boundary conditions. 
For small values of $\penaltyPushforward$ we observe a splitting of the bump maps in the corners with the outer one being blended out and the inner one being mainly transported into the middle, whereas 
for larger values of $\penaltyPushforward$ we observe a blending process without significant transport.
Furthermore, increasing the viscous dissipation parameter $\gamma$ leads as in Figure \ref{fig:differentViscosity} to a translation of the whole bump, while the mass overhead is continuously faded-out.
\begin{figure}[t]
\setlength{\unitlength}{.05\linewidth}
\resizebox{0.85\linewidth}{!}{
  \begin{minipage}[h]{1.0\textwidth}
  \begin{picture}(20,2.4)
    \put(0,0){\includegraphics[width=0.1\textwidth]{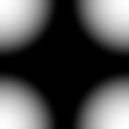}}
    \put(2.2,0){\includegraphics[width=0.1\textwidth]{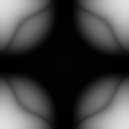}}
    \put(4.4,0){\includegraphics[width=0.1\textwidth]{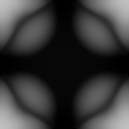}}
    \put(6.6,0){\includegraphics[width=0.1\textwidth]{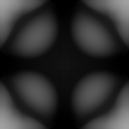}}
    \put(8.8, 0){\includegraphics[width=0.1\textwidth]{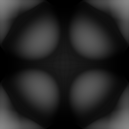}}
    \put(11,0){\includegraphics[width=0.1\textwidth]{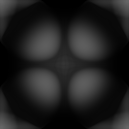}}
    \put(13.2,0){\includegraphics[width=0.1\textwidth]{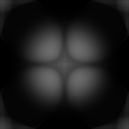}}
    \put(15.4,0){\includegraphics[width=0.1\textwidth]{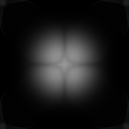}}
    \put(17.6,0){\includegraphics[width=0.1\textwidth]{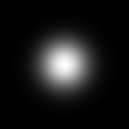}}
  \end{picture}
  \end{minipage}
}
\resizebox{0.85\linewidth}{!}{
  \begin{minipage}[h]{1.0\textwidth}
  \begin{picture}(20,2.4)
    \put(0,0){\includegraphics[width=0.1\textwidth]{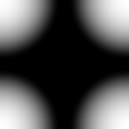}}
    \put(2.2,0){\includegraphics[width=0.1\textwidth]{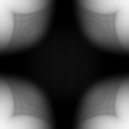}}
    \put(4.4,0){\includegraphics[width=0.1\textwidth]{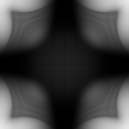}}
    \put(6.6,0){\includegraphics[width=0.1\textwidth]{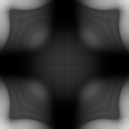}}
    \put(8.8, 0){\includegraphics[width=0.1\textwidth]{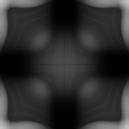}}
    \put(11,0){\includegraphics[width=0.1\textwidth]{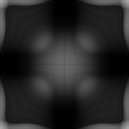}}
    \put(13.2,0){\includegraphics[width=0.1\textwidth]{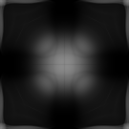}}
    \put(15.4,0){\includegraphics[width=0.1\textwidth]{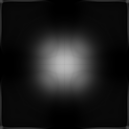}}
    \put(17.6,0){\includegraphics[width=0.1\textwidth]{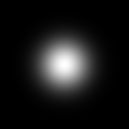}}
  \end{picture}
  \end{minipage}
}
\resizebox{0.85\linewidth}{!}{
  \begin{minipage}[h]{1.0\textwidth}
  \begin{picture}(20,2.4)
    \put(0,0){\includegraphics[width=0.1\textwidth]{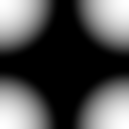}}
    \put(2.2,0){\includegraphics[width=0.1\textwidth]{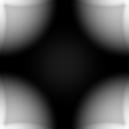}}
    \put(4.4,0){\includegraphics[width=0.1\textwidth]{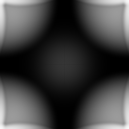}}
    \put(6.6,0){\includegraphics[width=0.1\textwidth]{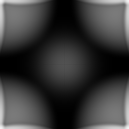}}
    \put(8.8, 0){\includegraphics[width=0.1\textwidth]{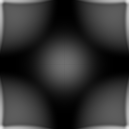}}
    \put(11,0){\includegraphics[width=0.1\textwidth]{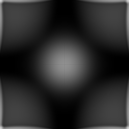}}
    \put(13.2,0){\includegraphics[width=0.1\textwidth]{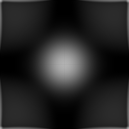}}
    \put(15.4,0){\includegraphics[width=0.1\textwidth]{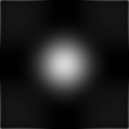}}
    \put(17.6,0){\includegraphics[width=0.1\textwidth]{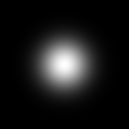}}
  \end{picture}
  \end{minipage}
}
\resizebox{0.85\linewidth}{!}{
  \begin{minipage}[h]{1.0\textwidth}
  \begin{picture}(20,2.4)
    \put(0.0,2.1){\line(1,0){19.6}}
    \put(0,0){\includegraphics[width=0.1\textwidth]{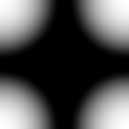}}
    \put(2.2,0){\includegraphics[width=0.1\textwidth]{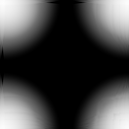}}
    \put(4.4,0){\includegraphics[width=0.1\textwidth]{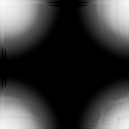}}
    \put(6.6,0){\includegraphics[width=0.1\textwidth]{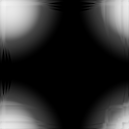}}
    \put(8.8, 0){\includegraphics[width=0.1\textwidth]{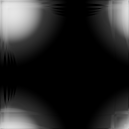}}
    \put(11,0){\includegraphics[width=0.1\textwidth]{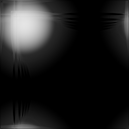}}
    \put(13.2,0){\includegraphics[width=0.1\textwidth]{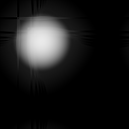}}
    \put(15.4,0){\includegraphics[width=0.1\textwidth]{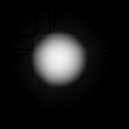}}
    \put(17.6,0){\includegraphics[width=0.1\textwidth]{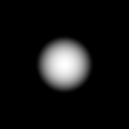}}
  \end{picture}
  \end{minipage}
}
\caption{Discrete geodesics ($K=9$) between two bump maps of different mass for different values of $\penaltyPushforward$ (from the first to the third row $\penaltyPushforward= 1,10,100$  with $\gamma=5\cdot10^{-4}$). 
In the fourth row the discrete geodesic for $\gamma=1$, $\penaltyPushforward=10^{-1}$ is displayed.
}
\label{fig:BenamouBrenierExampleDiffMassCompare}
\end{figure}
In Figure \ref{fig:differentMass} the input images consists of characteristic functions of two 
rectangles of different size.  Now, we impose natural boundary on $\partial \domain$.
For small $\penaltyPushforward$ and strong penalization of sources the surplus of mass is pushed outwards, whereas for large $\penaltyPushforward$ one observes a simple blending and almost no transport.
\begin{figure}[t]
\setlength{\unitlength}{.05\linewidth}
\resizebox{0.85\linewidth}{!}{
  \begin{minipage}[h]{1.0\textwidth}
  \begin{picture}(20,2.4)
        \put(0,0){\includegraphics[width=0.1\textwidth]{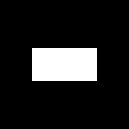}}
    \put(2.2,0){\includegraphics[width=0.1\textwidth]{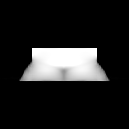}}
    \put(4.4,0){\includegraphics[width=0.1\textwidth]{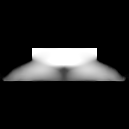}}
    \put(6.6,0){\includegraphics[width=0.1\textwidth]{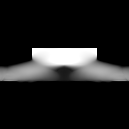}}
    \put(8.8, 0){\includegraphics[width=0.1\textwidth]{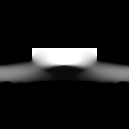}}
    \put(11,0){\includegraphics[width=0.1\textwidth]{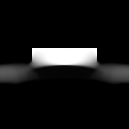}}
    \put(13.2,0){\includegraphics[width=0.1\textwidth]{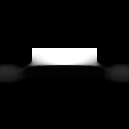}}
    \put(15.4,0){\includegraphics[width=0.1\textwidth]{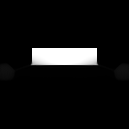}}
    \put(17.6,0){\includegraphics[width=0.1\textwidth]{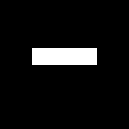}}
  \end{picture}
  \end{minipage}
}
\resizebox{0.85\linewidth}{!}{
  \begin{minipage}[h]{1.0\textwidth}
  \begin{picture}(20,2.4)
        \put(0,0){\includegraphics[width=0.1\textwidth]{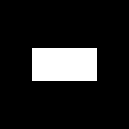}}
    \put(2.2,0){\includegraphics[width=0.1\textwidth]{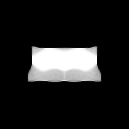}}
    \put(4.4,0){\includegraphics[width=0.1\textwidth]{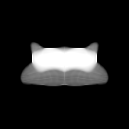}}
    \put(6.6,0){\includegraphics[width=0.1\textwidth]{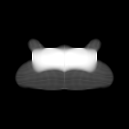}}
    \put(8.8, 0){\includegraphics[width=0.1\textwidth]{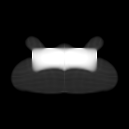}}
    \put(11,0){\includegraphics[width=0.1\textwidth]{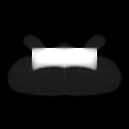}}
    \put(13.2,0){\includegraphics[width=0.1\textwidth]{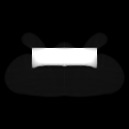}}
    \put(15.4,0){\includegraphics[width=0.1\textwidth]{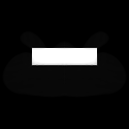}}
    \put(17.6,0){\includegraphics[width=0.1\textwidth]{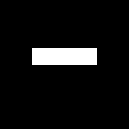}}
  \end{picture}
  \end{minipage}
}
\resizebox{0.85\linewidth}{!}{
  \begin{minipage}[h]{1.0\textwidth}
  \begin{picture}(20,2.4)
        \put(0,0){\includegraphics[width=0.1\textwidth]{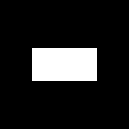}}
    \put(2.2,0){\includegraphics[width=0.1\textwidth]{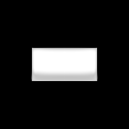}}
    \put(4.4,0){\includegraphics[width=0.1\textwidth]{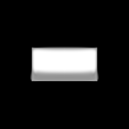}}
    \put(6.6,0){\includegraphics[width=0.1\textwidth]{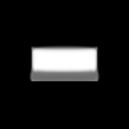}}
    \put(8.8, 0){\includegraphics[width=0.1\textwidth]{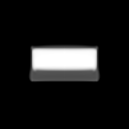}}
    \put(11,0){\includegraphics[width=0.1\textwidth]{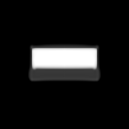}}
    \put(13.2,0){\includegraphics[width=0.1\textwidth]{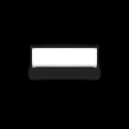}}
    \put(15.4,0){\includegraphics[width=0.1\textwidth]{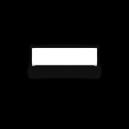}}
    \put(17.6,0){\includegraphics[width=0.1\textwidth]{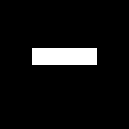}}
  \end{picture}
  \end{minipage}
}
\caption{Discrete geodesics between two rectangles of different mass for different values of $\penaltyPushforward$. Top: $\penaltyPushforward= 10^{-2}$, middle: $\penaltyPushforward = 10^{-1}$, bottom: $\penaltyPushforward =1$
($K=9$, $\gamma=10^{-2}$).}
\label{fig:differentMass}
\end{figure}

Furthermore, Figure \ref{fig:blending} compares our model with the metamorphosis model on the discrete geodesic between two images consisting of a light and a dark square and the flipped configuration. 
For very small density modulation parameter ($\penaltyPushforward=0.01$) we observe a transport of a ''light block" from the bottom to the top square, especially mass is approximately preserved.
In case of the metamorphism model with purely viscous flow, we see a transport of the lighter square combined with a fading in and out of the darker phase. 
\begin{figure}[!htbp]
\setlength{\unitlength}{.05\linewidth}
\resizebox{0.85\linewidth}{!}{
  \begin{minipage}[h]{1.0\textwidth}
  \begin{picture}(20,2.4)
    \put(0,0){\includegraphics[width=0.1\textwidth]{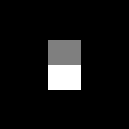}}
    \put(2.2,0){\includegraphics[width=0.1\textwidth]{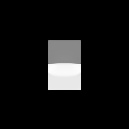}}
    \put(4.4,0){\includegraphics[width=0.1\textwidth]{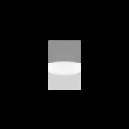}}
    \put(6.6,0){\includegraphics[width=0.1\textwidth]{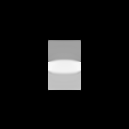}}
    \put(8.8, 0){\includegraphics[width=0.1\textwidth]{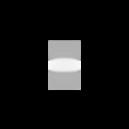}}
    \put(11,0){\includegraphics[width=0.1\textwidth]{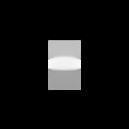}}
    \put(13.2,0){\includegraphics[width=0.1\textwidth]{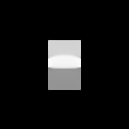}}
    \put(15.4,0){\includegraphics[width=0.1\textwidth]{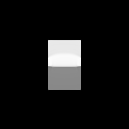}}
    \put(17.6,0){\includegraphics[width=0.1\textwidth]{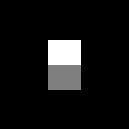}}
  \end{picture}
  \end{minipage}
}
\resizebox{0.85\linewidth}{!}{
  \begin{minipage}[h]{1.0\textwidth}
  \begin{picture}(20,2.4)
    \put(0,0){\includegraphics[width=0.1\textwidth]{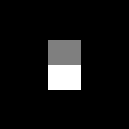}}
    \put(2.2,0){\includegraphics[width=0.1\textwidth]{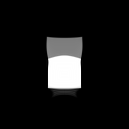}}
    \put(4.4,0){\includegraphics[width=0.1\textwidth]{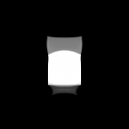}}
    \put(6.6,0){\includegraphics[width=0.1\textwidth]{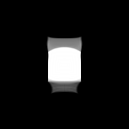}}
    \put(8.8, 0){\includegraphics[width=0.1\textwidth]{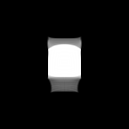}}
    \put(11,0){\includegraphics[width=0.1\textwidth]{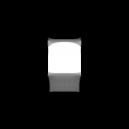}}
    \put(13.2,0){\includegraphics[width=0.1\textwidth]{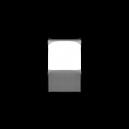}}
    \put(15.4,0){\includegraphics[width=0.1\textwidth]{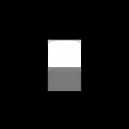}}
    \put(17.6,0){\includegraphics[width=0.1\textwidth]{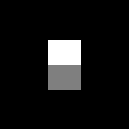}}
  \end{picture}
  \end{minipage}
}
\caption{Comparison of the combined model with viscosity parameter $\gamma=1$ (top) and the metamorphosis model (bottom) for $\penaltyPushforward=10^{-2}$. }
\label{fig:blending}
\end{figure}

As a first imaging application we pick up in Figure \ref{fig:MongeKantorovich}  an example from \cite{papadakis2014optimal}. For small $\penaltyPushforward$ and small $\gamma$ we obtain a very similar result. Finally, in Figure \ref{fig:MRIs} the geodesic between two different slices of the same human brain recorder via MRI is shown. The corresponding image intensities are characterized by substantially different masses. In fact, it is the incorporation of both the source term and the viscous dissipation term which enables a reasonable morph between the two slices. Thereby, the source terms allows for local image intensity modulation, whereas the viscous dissipation ensures regularity of the resulting transport path.
\begin{figure}[t]
\setlength{\unitlength}{.05\linewidth}
\resizebox{0.85\linewidth}{!}{
  \begin{picture}(20,3)
        \put(0,0){\includegraphics[width=0.1\textwidth]{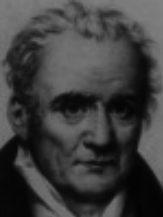}}
    \put(2.2,0){\includegraphics[width=0.1\textwidth]{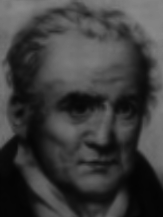}}
    \put(4.4,0){\includegraphics[width=0.1\textwidth]{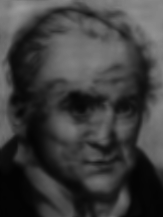}}
    \put(6.6,0){\includegraphics[width=0.1\textwidth]{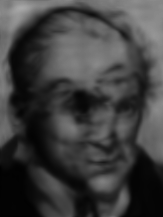}}
    \put(8.8,0){\includegraphics[width=0.1\textwidth]{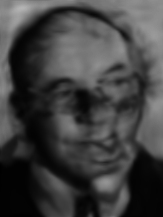}}
    \put(11,0){\includegraphics[width=0.1\textwidth]{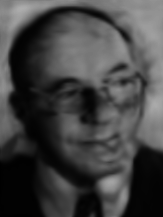}}
    \put(13.2,0){\includegraphics[width=0.1\textwidth]{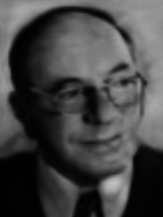}}
    \put(15.4,0){\includegraphics[width=0.1\textwidth]{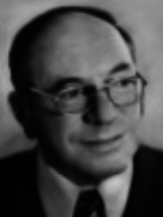}}
    \put(17.6,0){\includegraphics[width=0.1\textwidth]{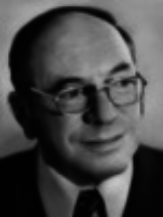}}
  \end{picture}
}
\caption{A discrete geodesic between images of Monge and Kantorovich with $\penaltyPushforward=10^{-2}$, $\gamma=10^{-2}$ (image provided by G. Peyr{\'e}).}
\label{fig:MongeKantorovich}
\end{figure}

\setlength{\unitlength}{0.05\textwidth}
\begin{figure}[t]
\resizebox{0.85\linewidth}{!}{
  \begin{picture}(20,9.4)
        \put(0,7){\includegraphics[width=0.1\textwidth]{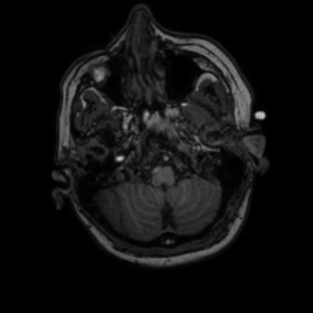}}
    \put(1.4,7.1){ \color{white}{ \bf $\u_0$ \bf} }
    \put(2.2,7){\includegraphics[width=0.1\textwidth]{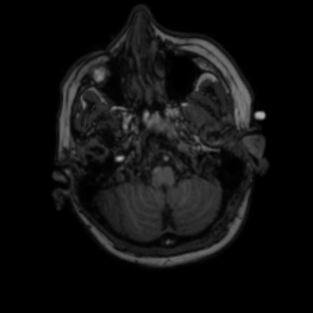}}
    \put(3.6,7.1){ \color{white}{ \bf $\u_1$ \bf} }
    \put(4.4,7){\includegraphics[width=0.1\textwidth]{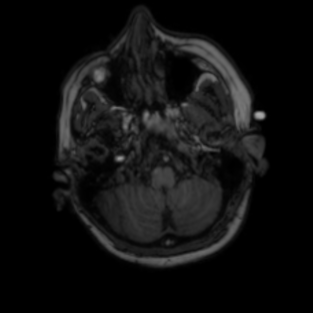}}
    \put(5.8,7.1){ \color{white}{ \bf $\u_2$ \bf} }
    \put(6.6,7){\includegraphics[width=0.1\textwidth]{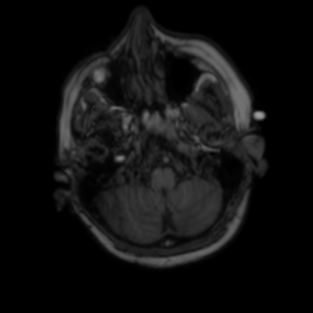}}
    \put(8.0,7.1){ \color{white}{ \bf $\u_31$ \bf} }
    \put(8.8,7){\includegraphics[width=0.1\textwidth]{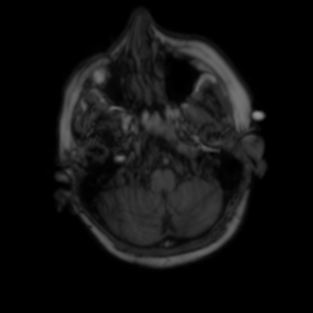}}
    \put(10.2,7.1){ \color{white}{ \bf $\u_4$ \bf} }
    \put(11.0,7){\includegraphics[width=0.1\textwidth]{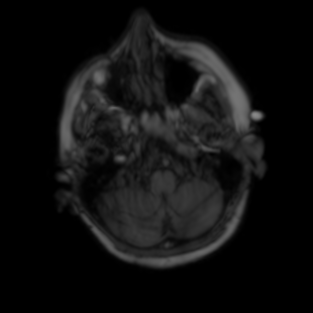}}
    \put(12.4,7.1){ \color{white}{ \bf $\u_5$ \bf} }
    \put(13.2,7){\includegraphics[width=0.1\textwidth]{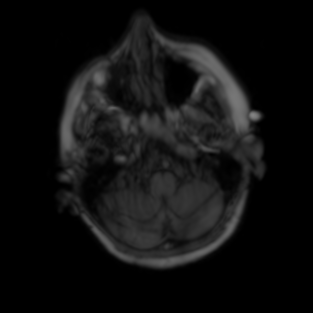}}
    \put(14.6,7.1){ \color{white}{ \bf $\u_6$ \bf} }
    \put(15.4,7){\includegraphics[width=0.1\textwidth]{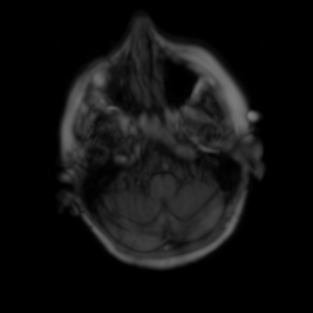}}
    \put(16.8,7.1){ \color{white}{ \bf $\u_7$ \bf} }
    \put(17.6,7){\includegraphics[width=0.1\textwidth]{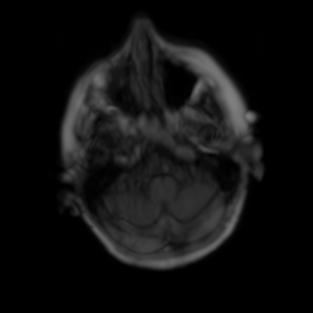}}
    \put(19.0,7.1){ \color{white}{ \bf $\u_8$ \bf} }
    \put(0,4.8){\includegraphics[width=0.1\textwidth]{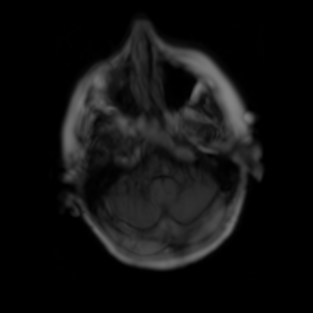}}
    \put(1.4,4.9){ \color{white}{ \bf $\u_9$ \bf} }
    \put(2.2,4.8){\includegraphics[width=0.1\textwidth]{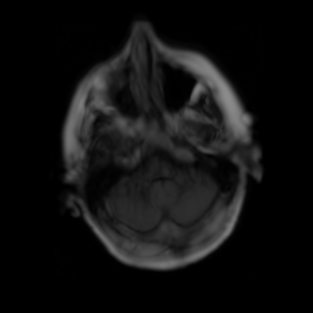}}
    \put(3.4,4.9){ \color{white}{ \bf $\u_{10}$ \bf} }
    \put(4.4,4.8){\includegraphics[width=0.1\textwidth]{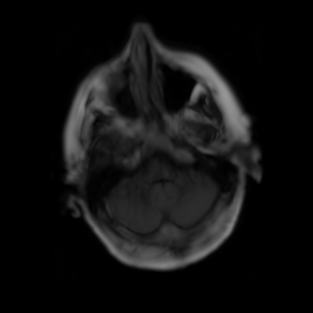}}
    \put(5.6,4.9){ \color{white}{ \bf $\u_{11}$ \bf} }
    \put(6.6,4.8){\includegraphics[width=0.1\textwidth]{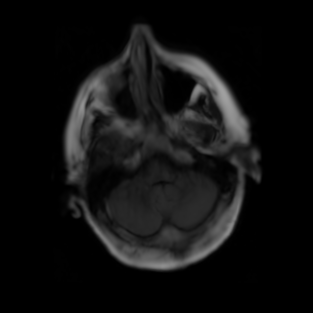}}
    \put(7.8,4.9){ \color{white}{ \bf $\u_{12}$ \bf} }
    \put(8.8,4.8){\includegraphics[width=0.1\textwidth]{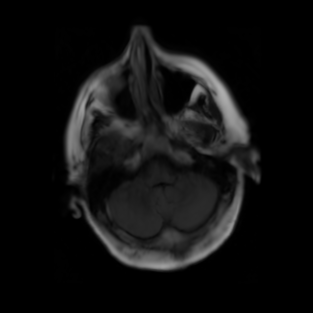}}
    \put(10,4.9){ \color{white}{ \bf $\u_{13}$ \bf} }
    \put(11.0,4.8){\includegraphics[width=0.1\textwidth]{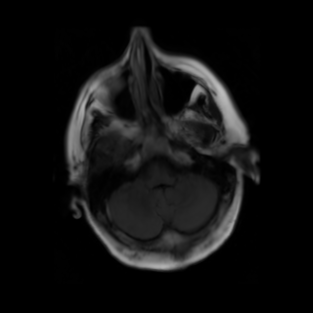}}
    \put(12.2,4.9){ \color{white}{ \bf $\u_{14}$ \bf} }
    \put(13.2,4.8){\includegraphics[width=0.1\textwidth]{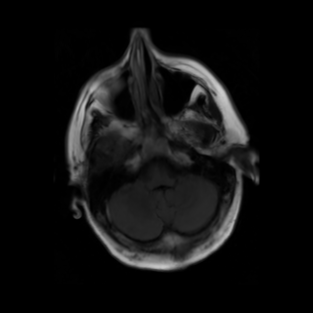}}
    \put(14.4,4.9){ \color{white}{ \bf $\u_{15}$ \bf} }
    \put(15.4,4.8){\includegraphics[width=0.1\textwidth]{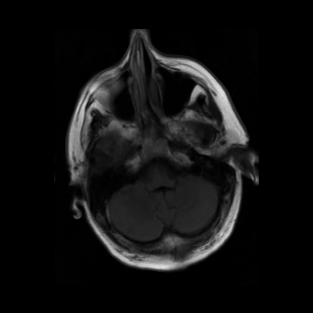}}
    \put(16.6,4.9){ \color{white}{ \bf $\u_{16}$ \bf} }
        \put(2.2,2.2){\includegraphics[width=0.1\textwidth]{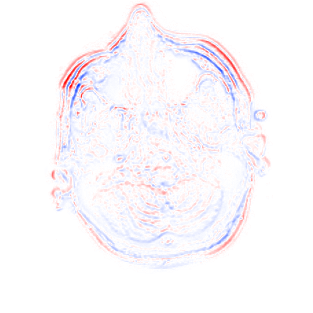}}
    \put(3.6,2.3){ \color{black}{ \bf $z_1$ \bf} }
    \put(4.4,2.2){\includegraphics[width=0.1\textwidth]{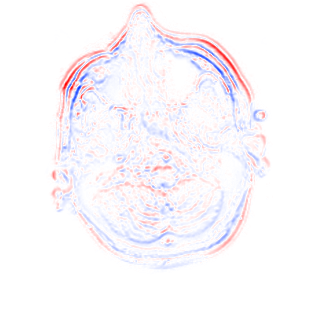}}
    \put(5.8,2.3){ \color{black}{ \bf $z_2$ \bf} }
    \put(6.6,2.2){\includegraphics[width=0.1\textwidth]{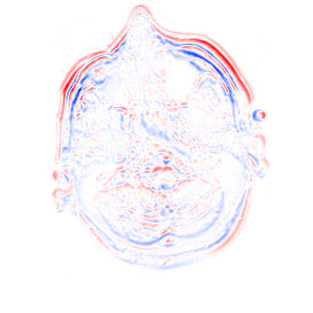}}
    \put(8.0,2.3){ \color{black}{ \bf $z_3$ \bf} }
    \put(8.8,2.2){\includegraphics[width=0.1\textwidth]{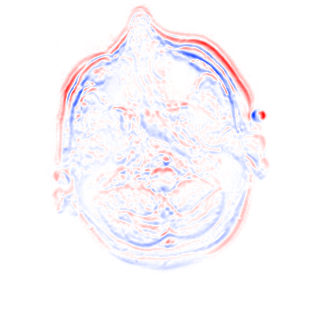}}
    \put(10.2,2.3){ \color{black}{ \bf $z_4$ \bf} }
    \put(11,2.2){\includegraphics[width=0.1\textwidth]{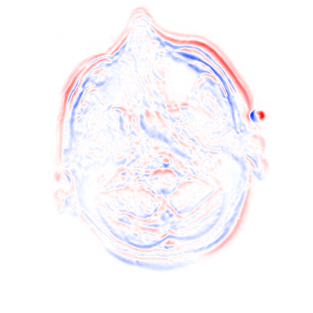}}
    \put(12.4,2.3){ \color{black}{ \bf $z_5$ \bf} }
    \put(13.2,2.2){\includegraphics[width=0.1\textwidth]{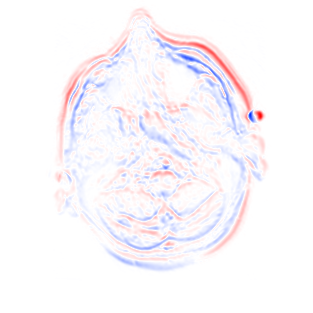}}
    \put(14.6,2.3){ \color{black}{ \bf $z_6$ \bf} }
    \put(15.4,2.2){\includegraphics[width=0.1\textwidth]{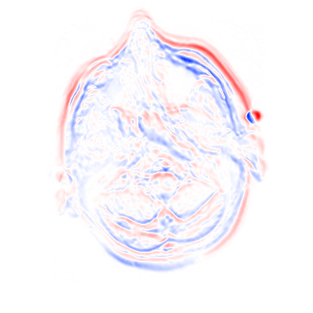}}
    \put(16.8,2.3){ \color{black}{ \bf $z_7$ \bf} }
    \put(17.6,2.2){\includegraphics[width=0.1\textwidth]{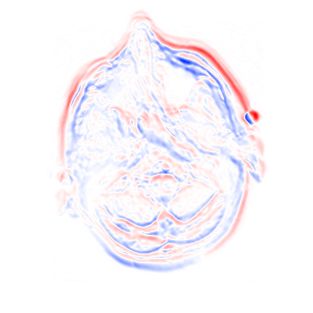}}
    \put(19.0,2.3){ \color{black}{ \bf $z_8$ \bf} }
    \put(0.0,0){\includegraphics[width=0.1\textwidth]{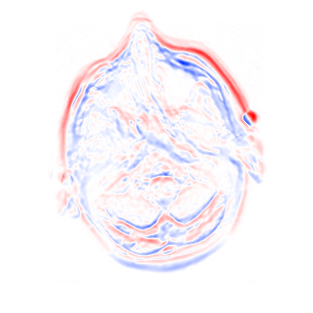}}
    \put(1.4, 0.1){ \color{black}{ \bf $z_9$ \bf} }
    \put(2.2,0){\includegraphics[width=0.1\textwidth]{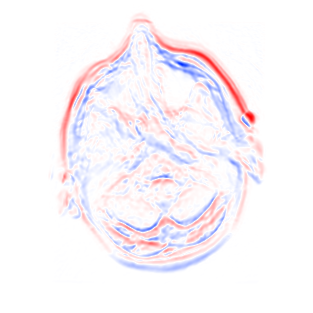}}
    \put(3.6, 0.1){ \color{black}{ \bf $z_{10}$ \bf} }
    \put(4.4,0){\includegraphics[width=0.1\textwidth]{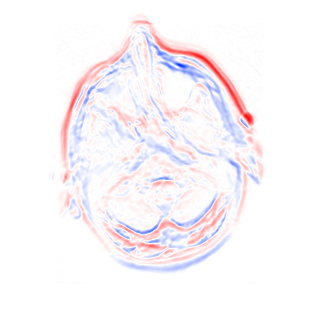}}
    \put(5.8, 0.1){ \color{black}{ \bf $z_{11}$ \bf} }
    \put(6.6,0){\includegraphics[width=0.1\textwidth]{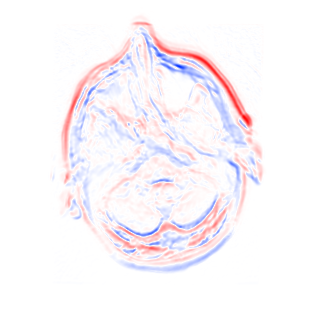}}
    \put(8, 0.1){ \color{black}{ \bf $z_{12}$ \bf} }
    \put(8.8,0){\includegraphics[width=0.1\textwidth]{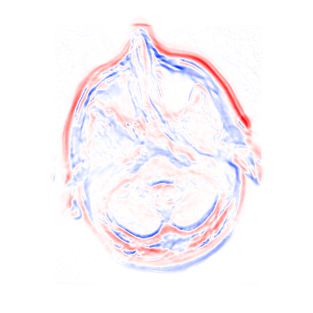}}
    \put(10.2, 0.1){ \color{black}{ \bf $z_{13}$ \bf} }
    \put(11.0,0){\includegraphics[width=0.1\textwidth]{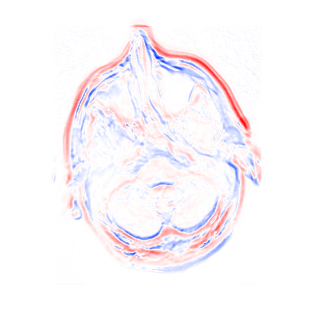}}
    \put(12.4, 0.1){ \color{black}{ \bf $z_{14}$ \bf} }
    \put(13.2,0){\includegraphics[width=0.1\textwidth]{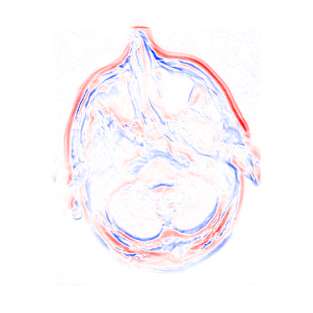}}
    \put(14.6, 0.1){ \color{black}{ \bf $z_{15}$ \bf} }
    \put(15.4,0){\includegraphics[width=0.1\textwidth]{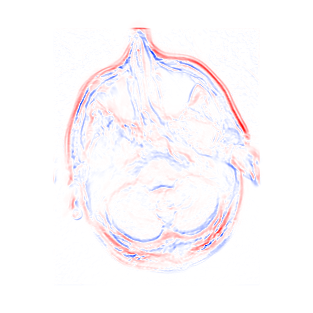}}
    \put(16.8, 0.1){ \color{black}{ \bf $z_{16}$ \bf} }
  \end{picture}
}
\caption{Two slices of the same 3D MRI data set of a human brain are connected with a discrete geodesic (data courtesy of H. Urbach, Neuroradiology, University Hospital Bonn). Top: discrete geodesic with $\penaltyPushforward=10^{-2}$, $\gamma=10^{-1}$, bottom: corresponding values of $z_k = \det(D\phi_k) \u_k \circ \phi_k - \u_{k-1}$ (blue: positive, red: negative).}
\label{fig:MRIs}
\end{figure}
There is  no guarantee that the alternating algorithm converges. To demonstrate the experimental convergence behaviour we choose the application shown in Fig.  \ref{fig:MRIs} and show the evolution of the $l^2$ norm of the difference between consequitive intensities in Fig. \ref{fig:errorplot}.
\begin{figure}[htbp!]
\resizebox{0.2\linewidth}{!}{
  \begin{minipage}[h]{0.5\textwidth}
 \begin{tikzpicture}
    \begin{loglogaxis}[
        xlabel=\textsc{iteration},
        ylabel=$\| \u^j - \u^{j-1} \|_{l^2}$
    ]
    \axispath\draw
            (7.49165,-10.02171)
        |-  (8.31801,-11.32467)
        node[near start,left] {$\frac{dy}{dx} = -1.58$};
      \addplot plot coordinates {
	  (1, 0.00611657)
	  (2, 0.000904781)
	  (3, 0.00048791)
	  (4, 0.000365233)
	  (5, 0.000221409)
	  (6, 0.000212917)
	  (7, 0.000227335)
	  (8, 5.89624e-05)
	  (9, 3.7831e-05)
	  (10, 4.11086e-05)
	  (11, 3.61958e-05)
	  (12, 2.31803e-05)
	  (13, 6.92171e-05)
	  (14, 1.87252e-05)
	  (15, 3.13328e-05)
    };
    \legend{$K=2$}
    \end{loglogaxis}
\end{tikzpicture}
\end{minipage}
}
\resizebox{0.2\linewidth}{!}{
  \begin{minipage}[h]{0.5\textwidth}
\begin{tikzpicture}
    \begin{loglogaxis}[
        xlabel=\textsc{iteration},
    ]
    \axispath\draw
            (7.49165,-10.02171)
        |-  (8.31801,-11.32467)
        node[near start,left] {$\frac{dy}{dx} = -1.58$};
      \addplot plot coordinates {
	  (1, 0.010785)
	  (2, 0.00124814)
	  (3, 0.000774907)
	  (4, 0.000305444)
	  (5, 0.000446813)
	  (6, 0.000293067)
	  (7, 0.000118888)
	  (8, 8.22818e-05)
	  (9, 9.49056e-05)
	  (10, 0.000100292)
	  (11, 0.000294279)
	  (12, 5.69816e-05)
	  (13, 0.000163001)
	  (14, 4.80309e-05)
	  (15, 6.25742e-05)
    };
    \legend{$K=4$}
    \end{loglogaxis}
\end{tikzpicture}
\end{minipage}
}
\resizebox{0.2\linewidth}{!}{
  \begin{minipage}[h]{0.5\textwidth}
\begin{tikzpicture}
    \begin{loglogaxis}[
        xlabel=\textsc{iteration},
    ]
    \axispath\draw
            (7.49165,-10.02171)
        |-  (8.31801,-11.32467)
        node[near start,left] {$\frac{dy}{dx} = -1.58$};
      \addplot plot coordinates {
	  (1, 0.0166897)
	  (2, 0.00190365)
	  (3, 0.000716768)
	  (4, 0.000835214)
	  (5, 0.000225586)
	  (6, 0.000159049)
	  (7, 9.73407e-05)
	  (8, 0.000135303)
	  (9, 9.02976e-05)
	  (10, 0.000157557)
	  (11, 0.000233976)
	  (12, 0.000135323)
	  (13, 7.31397e-05)
	  (14, 0.000180224)
	  (15, 9.03364e-05)
    };
    \legend{$K=8$}
    \end{loglogaxis}
\end{tikzpicture}
\end{minipage}
}
\resizebox{0.2\linewidth}{!}{
  \begin{minipage}[h]{0.5\textwidth}
\begin{tikzpicture}
    \begin{loglogaxis}[
        xlabel=\textsc{iteration},
    ]
    \axispath\draw
            (7.49165,-10.02171)
        |-  (8.31801,-11.32467)
        node[near start,left] {$\frac{dy}{dx} = -1.58$};
      \addplot plot coordinates {
	  (1, 0.0237109)
	  (2, 0.00296791)
	  (3, 0.000992212)
	  (4, 0.000368002)
	  (5, 0.000563277)
	  (6, 0.000201171)
	  (7, 0.00041147)
	  (8, 0.000252613)
	  (9, 0.00021584)
	  (10, 0.000290127)
	  (11, 0.000183656)
	  (12, 0.000180219)
	  (13, 0.000195304)
	  (14, 0.000181668)
	  (15, 0.000125673)   
};   
    \legend{$K=16$}
    \end{loglogaxis}
\end{tikzpicture}
\end{minipage}
}
\label{fig:errorplot}
\caption{The convergence of the alternating descent method is shown for the application in Fig. \ref{fig:MRIs}. For the different levels of the cascadic descent scheme ($K=2,4,8,16$) the $l^2$ norm of the difference of consequitive space time densities $\u^j$ is visualized using log-log plots. }
\end{figure}
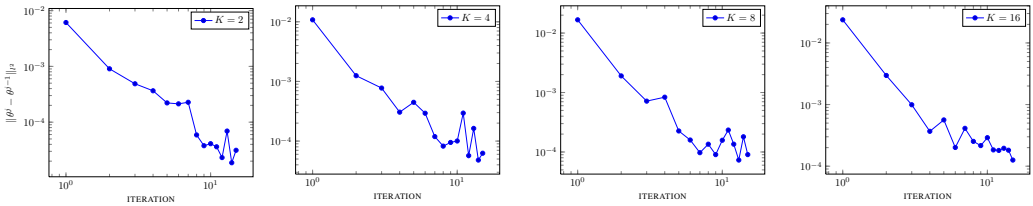

\FloatBarrier
\section{The Benamou-Brenier discretization for the non viscous model}\label{sec:BB}
In this section we numerically compare the proposed approach \eqref{eq:newtransport} with the numerical scheme for optimal transport proposed by Benamou and Brenier \cite{BeBr00}, where the 
mass constraint is relaxed.  After the change of variables  $(\u,v) \mapsto (\u,m = \u v)$  the minimization problem of the discrete path energy is rewritten as
\begin{align*}
 \sup\limits_{z} \min\limits_{\phi} \left(F(B\phi) + G(\phi) + \int_0^1 \int_\domain \phi \cdot z - \frac{1}{\penaltyPushforward} |z|^2 \diff x \diff t\right) \, ,
\end{align*}
where $\phi$ is a Lagrange multiplier introduced to satisfy the condition on $z$, $F$ is the indicator function of the convex set $K=\{ (a,b) \in \R \times \R^d \, : \, a + \frac{|b|^2}{2} \leq 0 \}$, 
$G(\phi) = \int_\domain \phi(0,\cdot) \u_0 - \phi(1,\cdot) \u_1 \diff x$, and $B: \phi \mapsto (\partial_t \phi, \nabla_x \phi)$. 
For the outer maximization in $z$ one gets the optimality condition $z = \frac{\penaltyPushforward}{2} \phi$. The augmented Lagrangian is given by
\begin{align*}
    L_r[\phi,q,\mu] 
 = F(q)+ G(\phi) + \int_0^1 \int_\domain \phi \cdot z - \frac{1}{\penaltyPushforward} |z|^2 + \mu \cdot \left( \nabla_{t,x} \phi - q \right)  + \frac{r}{2}  |\nabla_{t,x} \phi - q |^2 \diff x \diff t \, ,
\end{align*}
with variables $q=(a,b)$, $\mu=(\u,m)$ and Benamou and Brenier propose an alternating gradient descent to compute the saddle point.
Using the fact that $z=\frac{\penaltyPushforward}{2} \phi$, one updates $z$ and $\phi$ simultaneously solving
$
 -r \triangle_{t,x} \phi^n + \frac{\penaltyPushforward}{2} \phi^n  = \div_{t,x} ( \mu^n - r q^{n-1} )
$
with Neumann boundary conditions in time, \ie 
$r \partial_t \phi^n(0,\cdot) = \u_0 - \u^n(0,\cdot) + r a^{n-1}(0,\cdot)$, $r \partial_t \phi^n(1,\cdot) = \u_1 - \u^n(1,\cdot) + r a^{n-1}(1,\cdot)$.
Let us emphasize that in \cite{BeBr00} the second term on the left hand side which reflects the source term already appeared in the original scheme by Benamou and Brenier 
as a regularization term. 

To study the impact of the parameter $\penaltyPushforward$ we pick up the problem already presented in Fig. \ref{fig:differentViscosity}. 
Now, we choose two input bump maps of different mass. Figure \ref{fig:BenamouBrenierExampleDiffMass65} shows discrete geodesics for different $\penaltyPushforward$.
\begin{figure}[htbp!]
\setlength{\unitlength}{.05\linewidth}
\resizebox{0.85\linewidth}{!}{
  \begin{picture}(16,2.4)
    \put(0,0){\includegraphics[width=0.1\textwidth]{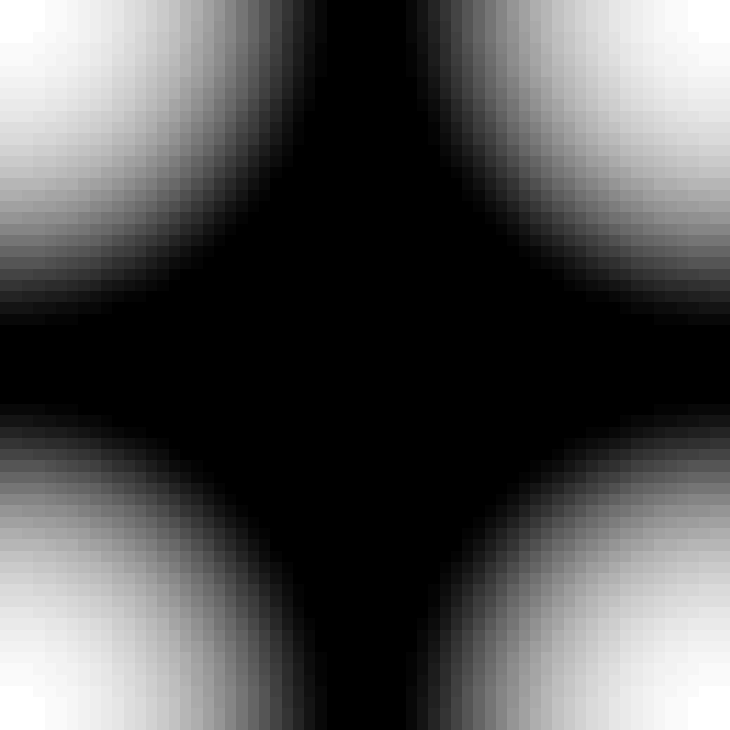}}
    \put(2.2,0){\includegraphics[width=0.1\textwidth]{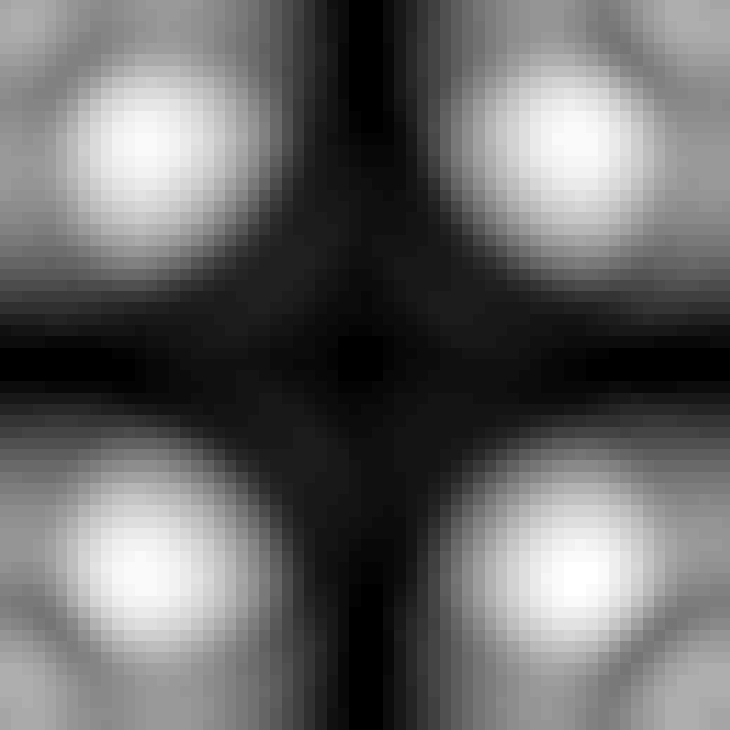}}
    \put(4.4,0){\includegraphics[width=0.1\textwidth]{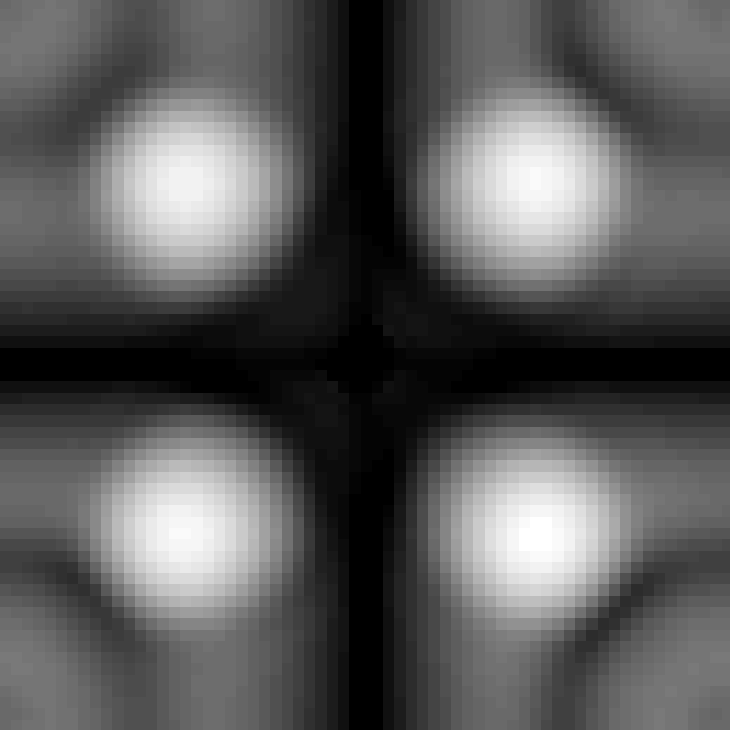}}
    \put(6.6,0){\includegraphics[width=0.1\textwidth]{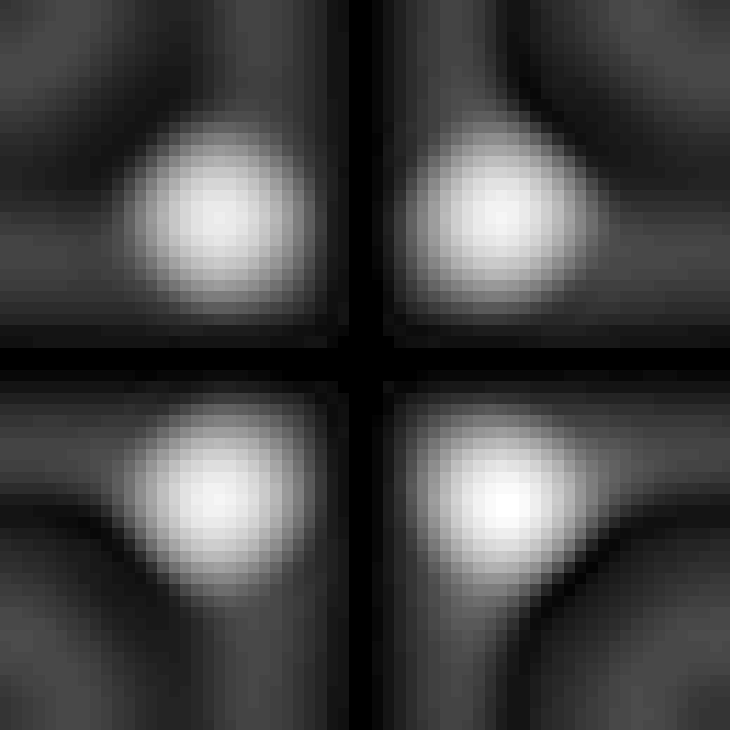}}
    \put(8.8,0){\includegraphics[width=0.1\textwidth]{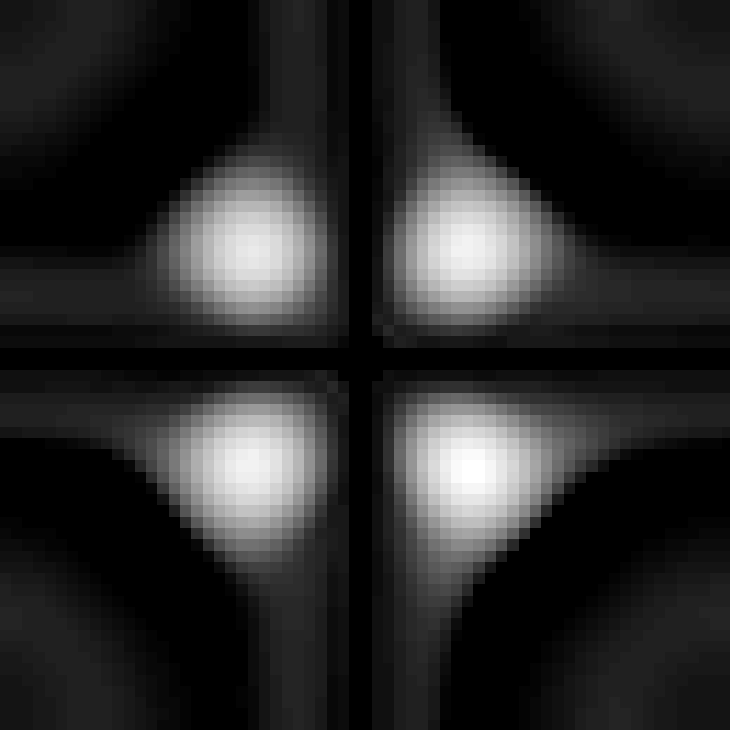}}
    \put(11,0){\includegraphics[width=0.1\textwidth]{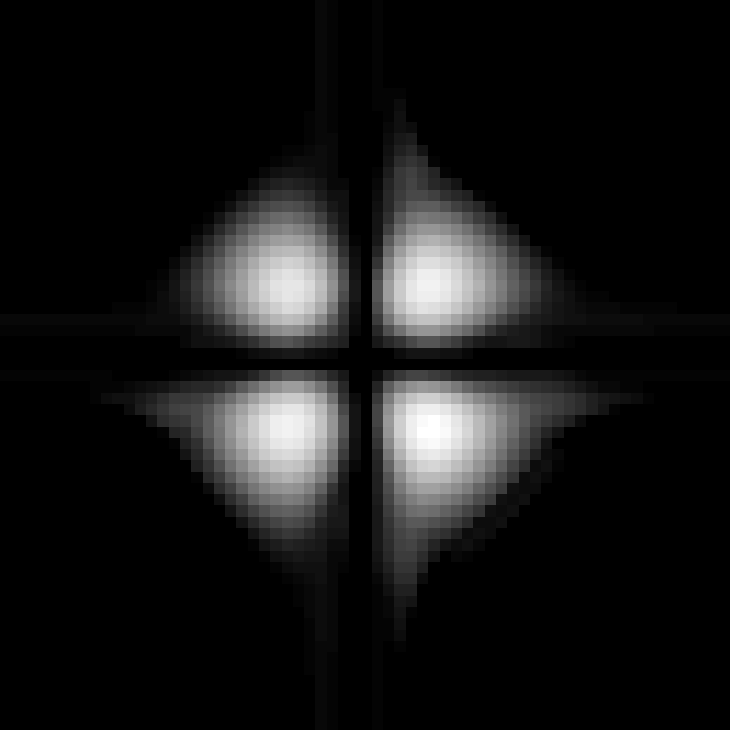}}
    \put(13.2,0){\includegraphics[width=0.1\textwidth]{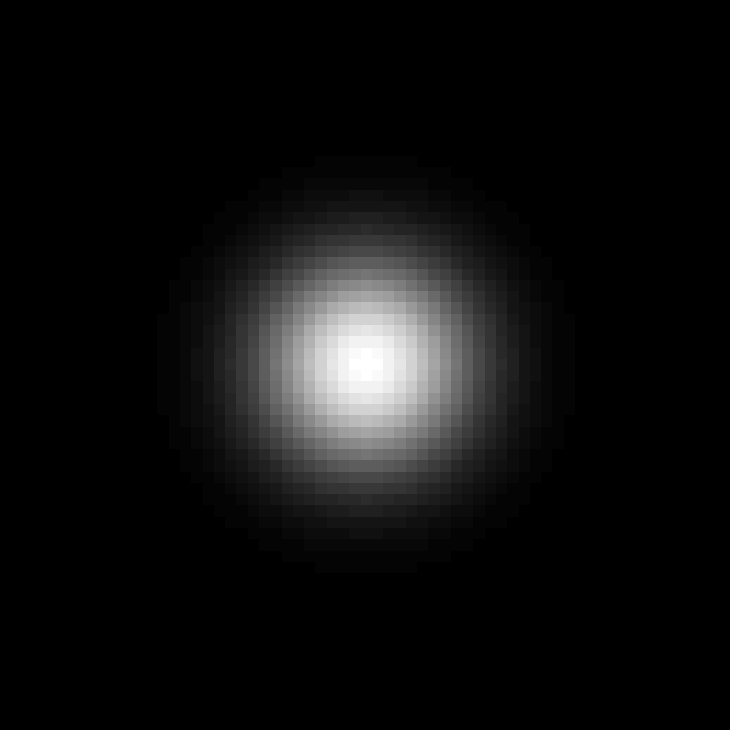}}
  \end{picture}
}
\resizebox{0.85\linewidth}{!}{
  \begin{picture}(16,2.4)
    \put(0,0){\includegraphics[width=0.1\textwidth]{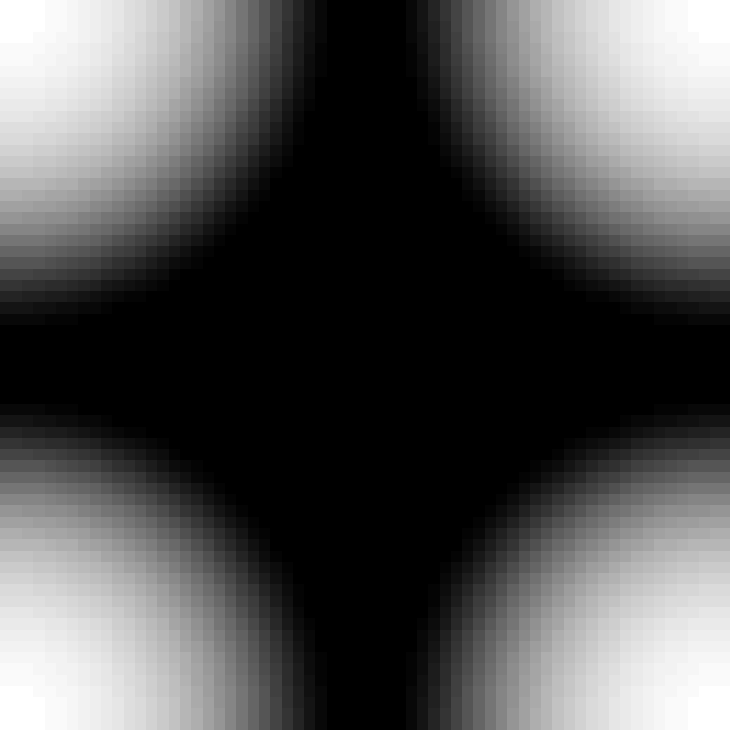}}
    \put(2.2,0){\includegraphics[width=0.1\textwidth]{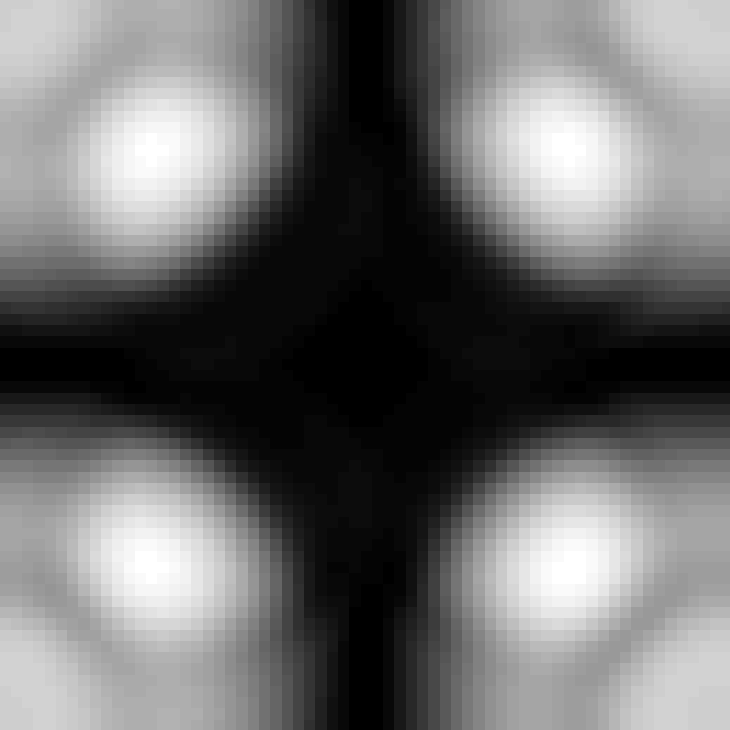}}
    \put(4.4,0){\includegraphics[width=0.1\textwidth]{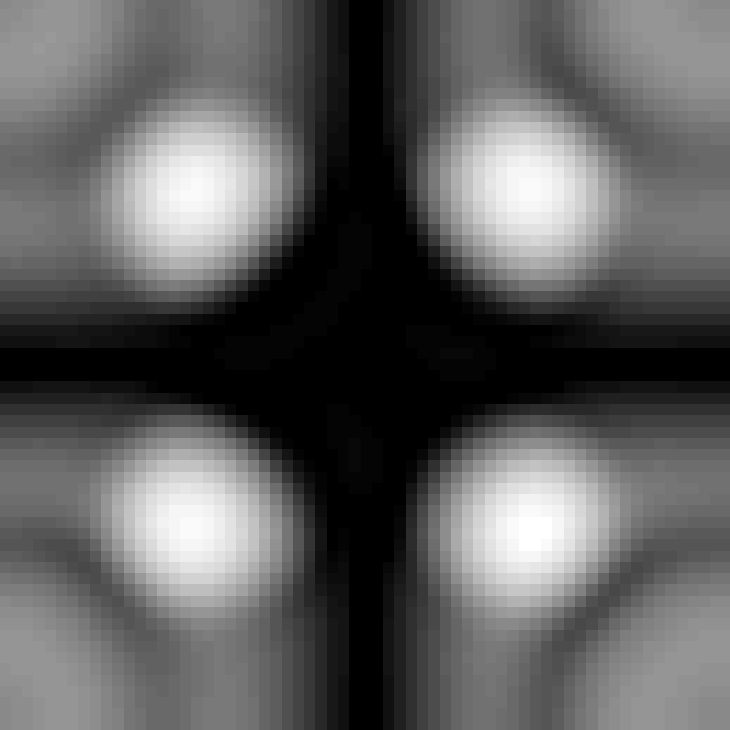}}
    \put(6.6,0){\includegraphics[width=0.1\textwidth]{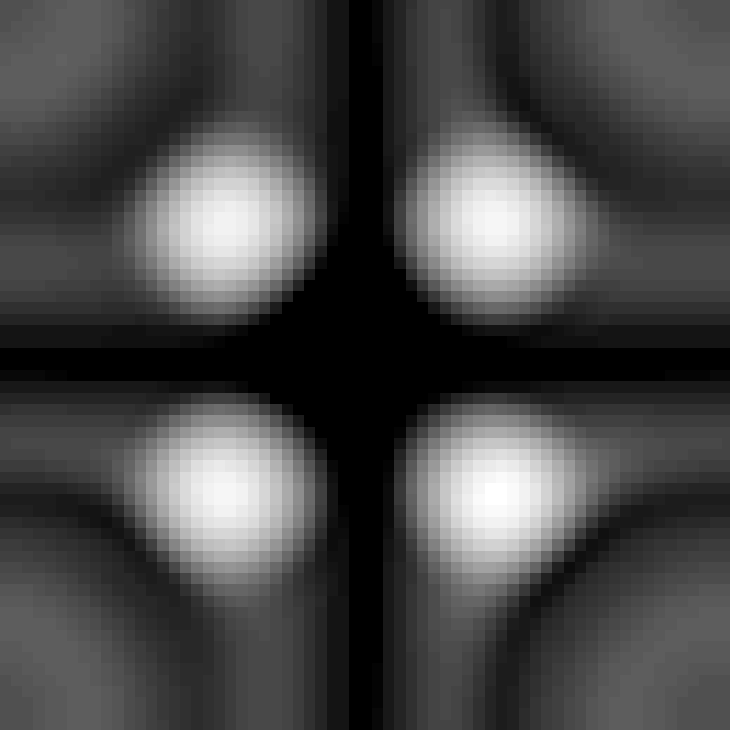}}
    \put(8.8,0){\includegraphics[width=0.1\textwidth]{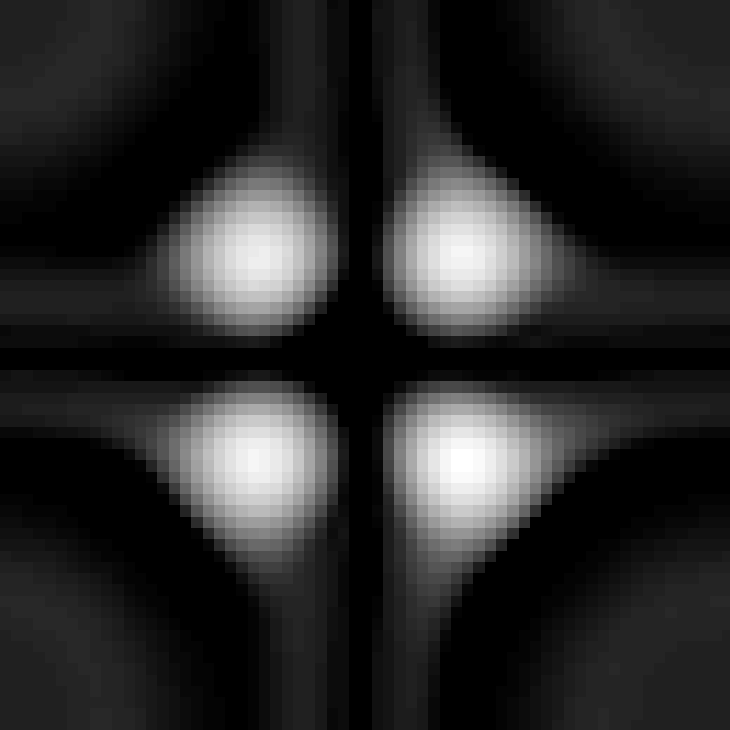}}
    \put(11,0){\includegraphics[width=0.1\textwidth]{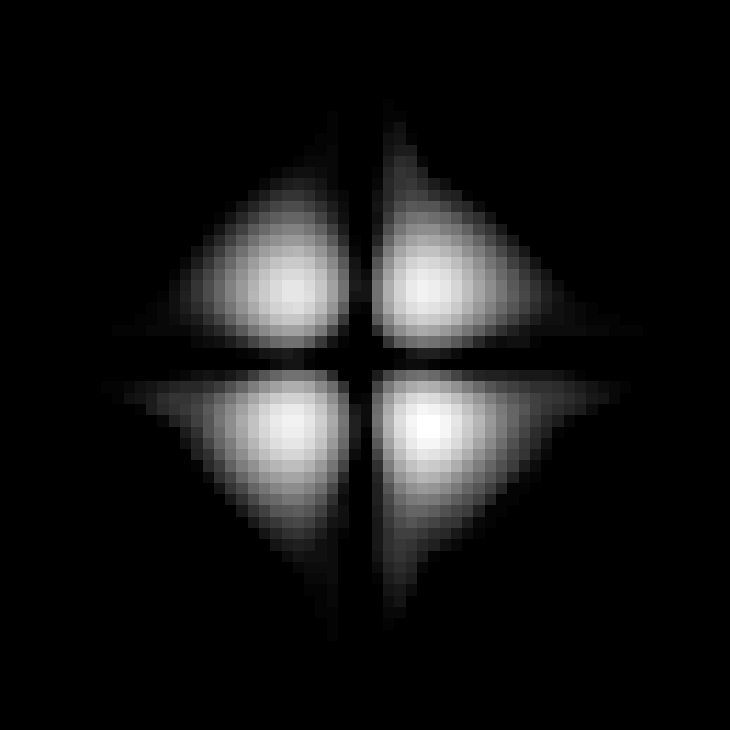}}
    \put(13.2,0){\includegraphics[width=0.1\textwidth]{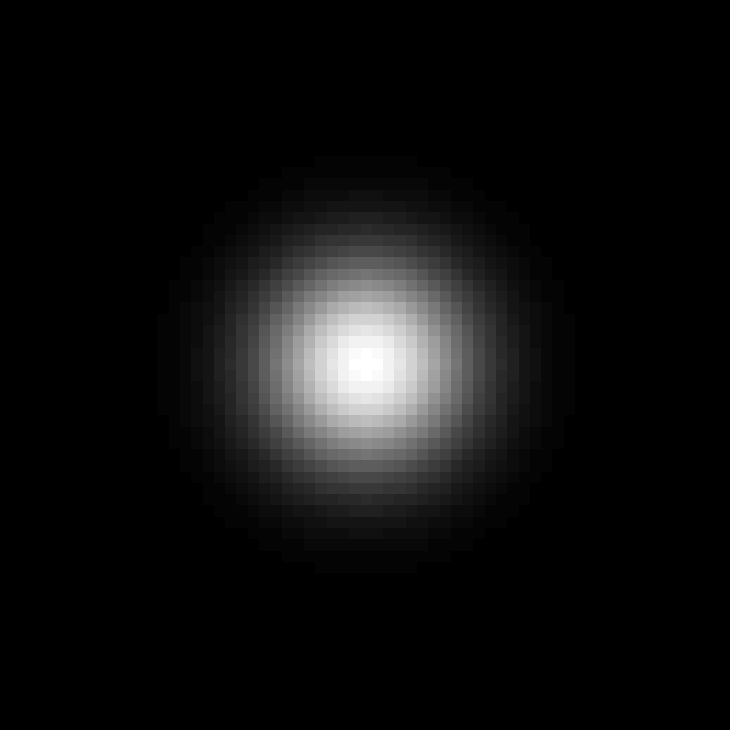}}
  \end{picture}
}
\resizebox{0.85\linewidth}{!}{
  \begin{picture}(16,2.4)
    \put(0,0){\includegraphics[width=0.1\textwidth]{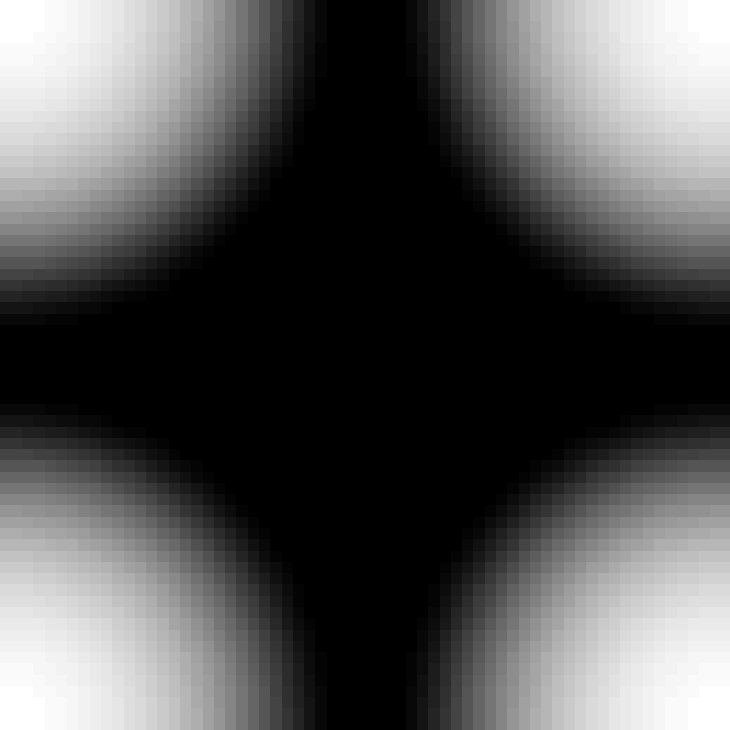}}
    \put(2.2,0){\includegraphics[width=0.1\textwidth]{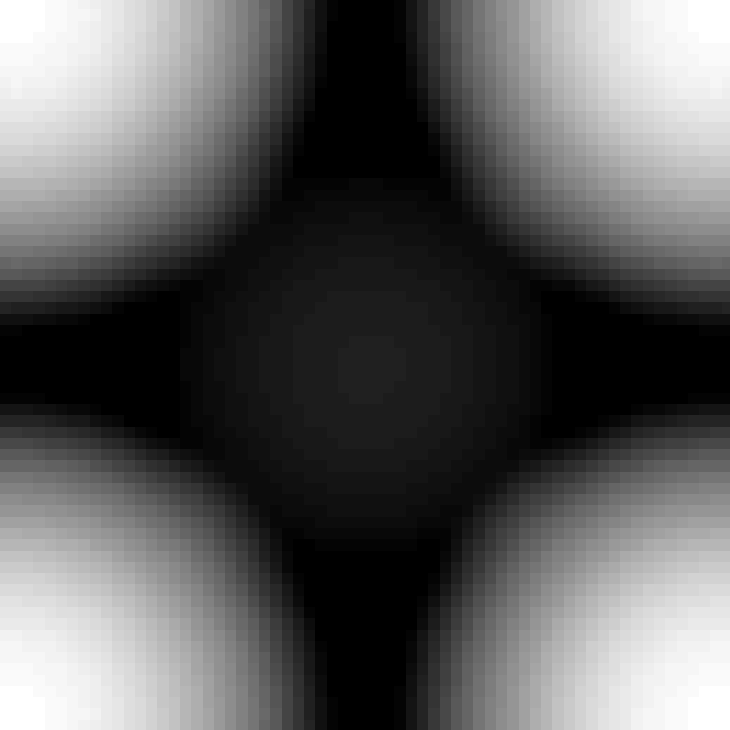}}
    \put(4.4,0){\includegraphics[width=0.1\textwidth]{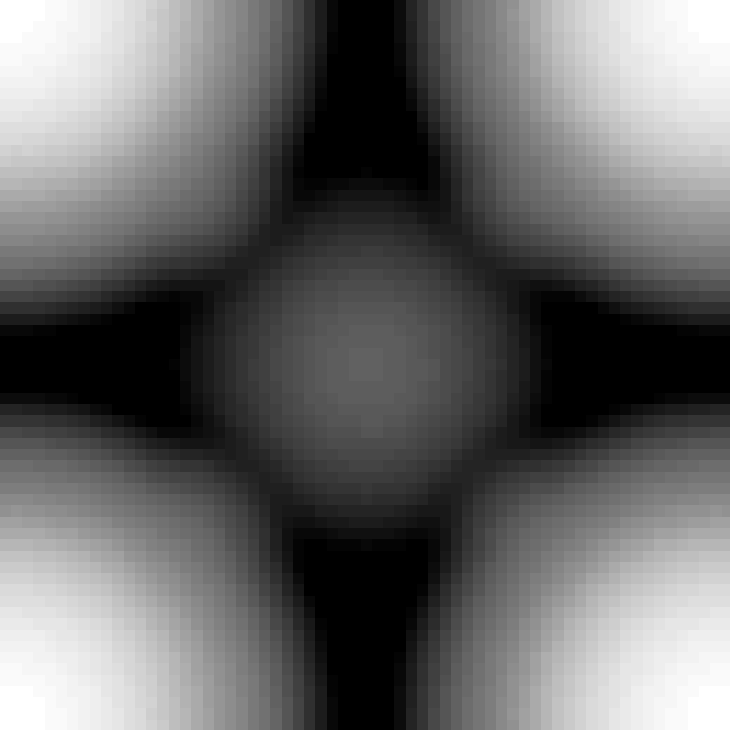}}
    \put(6.6,0){\includegraphics[width=0.1\textwidth]{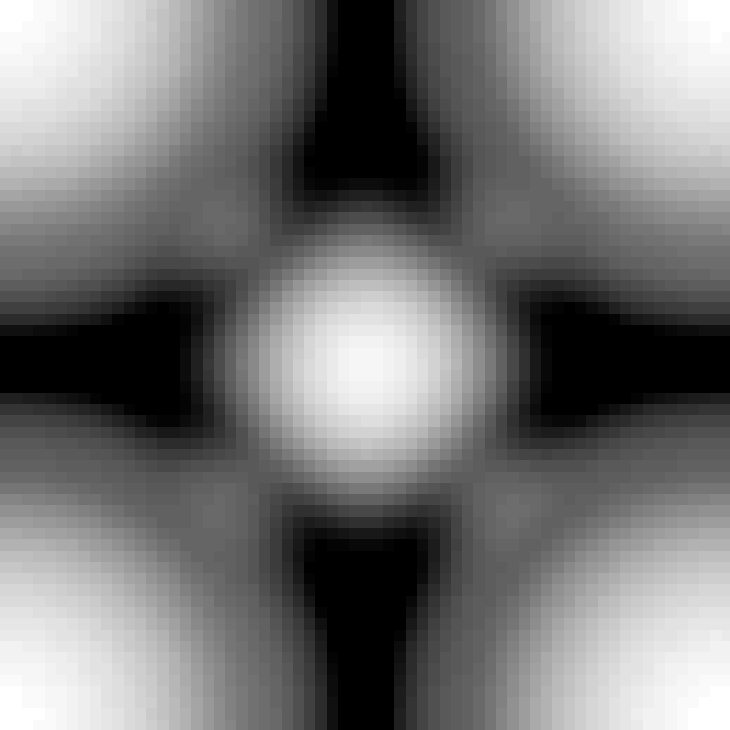}}
    \put(8.8,0){\includegraphics[width=0.1\textwidth]{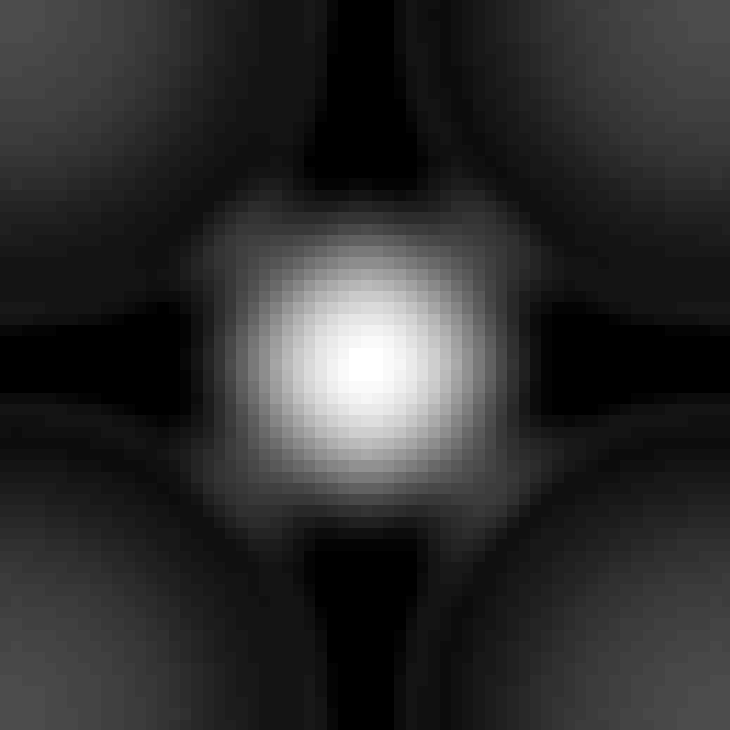}}
    \put(11,0){\includegraphics[width=0.1\textwidth]{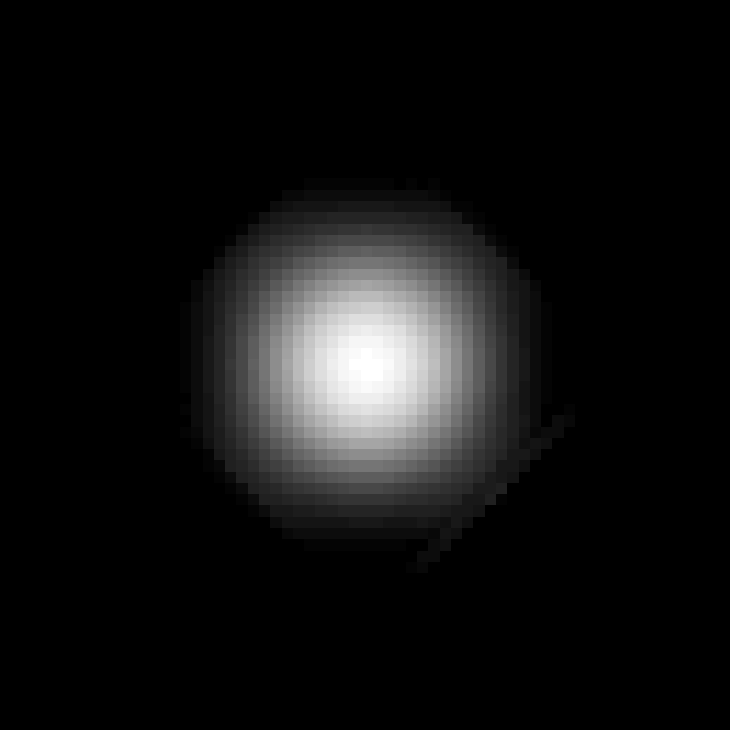}}
    \put(13.2,0){\includegraphics[width=0.1\textwidth]{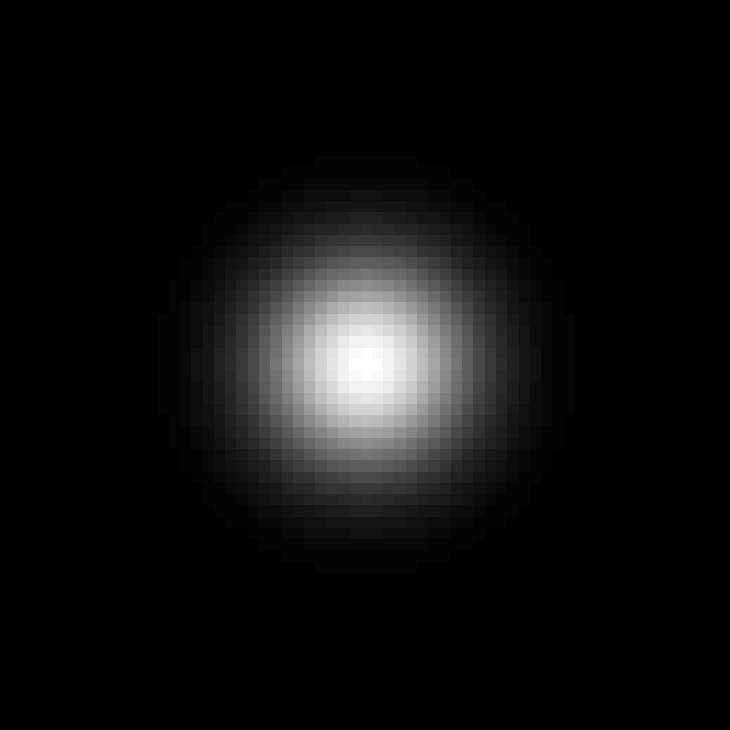}}
  \end{picture}
}
\caption{The Benamou-Brenier discretization applied to optimal transport with relaxed mass constraint for input images with bump maps of different mass.
The rows show equidistributed time steps for a time step size $\tau = \tfrac{1}{60}$ and different $\penaltyPushforward$ (top: $\penaltyPushforward = 1$, middle: $\penaltyPushforward = 10$, bottom: $\penaltyPushforward = 100$). }
\label{fig:BenamouBrenierExampleDiffMass65}
\end{figure}
For large $\penaltyPushforward$ we basically observe pure blending and almost no transport, whereas for smaller $\penaltyPushforward$ mass is first reduced for each bump map leading to a concentration in $4$ bumps 
which are then transported.  The differences to Fig. \ref{fig:BenamouBrenierExampleDiffMassCompare} seem to be due to the presence of still some viscous dissipation.
\section{Application of the variational time discretization to Riemannian barycentres}\label{sec:barycentre}

As a further application of our time discrete geodesics in the space of images we consider the computation of (weighted) discrete barycentres. 
We call $\u^\bary$ the barycentre of $M$ input images $\u^1,\ldots, \u^M$ for given weights $\lambda^1,\ldots,\lambda^M$ with 
$\lambda^m\geq 0$ and $\sum_{m=1}^M \lambda^m = 1$, if $\u^\bary$ minimizes 
\beqn\label{eq:bar}
\Bdg[\u] = \sum^M_{m=1} \lambda^m \Wdg[\u,\u^m]^2 
\nonumber
\eeqn
Next, replacing the time continuous path energy $\Wdg$ by the time discrete energy $\WdgK$ \eqref{eq:WdgK} we ask for a minimizer of the energy
\beqn\label{eq:barK}
\Bdg[\u] = \min_{\substack{(\u_k^m)_{k=0,\ldots, K} \subset \Imagespace \\(\phi_k^m)_{k=1,\ldots, K} \subset \deformationSpace \\ \u = \u^m_0}} \;\;\sum^M_{m=1} \lambda^m \EdgK[\u_0^m, \ldots, \u_K^m,\phi_1^m, \ldots, \phi_K^m]
\nonumber
\eeqn
over $M$ discrete image paths $(\u_k^m)_{k=0,\ldots, K}$  ($m=1,\ldots, M$) and $M$ discrete families  $(\phi_k^m)_{k=1,\ldots, K}$ ($m=1,\ldots, M$)
with the last image of the $m$ discrete image paths being the $m$th input image ($\u_K^m = \u^m$) and the additional constraint that
the set of first images being all equal to $\u$ ( $\u = \u^m_0$ for all $m=1,\ldots, M$).

The necessary conditions for the images $\u_k^m$ and the deformations $\phi_k^m$ ($k=1,\ldots, M$, $m=1,\ldots, M$) 
are identical to those for simple discrete geodesics connecting the corresponding pair of images $(\u, \u^m)$. Solely the condition for the barycentre image itself
changes to 
\beqn\label{eq:ELbary}
\u^\bary(x) =  \max\left(0, \sum_{m=1}^M \lambda^m \left(\det( D \phi_1^m ) \u_1^m(\phi_1^m(x)) - \frac{\penaltyPushforward}{2} |\phi_1^m(x) - \id|^2 \right)\right)\,.
\nonumber
\eeqn
Finally, we take into account the spatial discretization introduced in Section \ref{sec:SpatialDiscretization} and define $\U^\bary$ as the fully discrete, weighted barycentre of the  
input images $\U^1, \ldots, \U^M$,  if $\U^\bary$ minimizes the energy 
\beqan\label{eq:bardiscrete}
\BdgKh[\u] &=& \sum^M_{m=1} \lambda^m \WdgKh[\U,\U^m]^2\\
&=&  \min_{\substack{(\U_k^m)_{k=0,\ldots, K} \subset \Imagespace \\(\Phi_k^m)_{k=1,\ldots, K} \subset \deformationSpace \\ \U = \U^m_0}} \; \sum^M_{m=1} \lambda^m \EdgKh[\U_0^m, \ldots, \U_K^m,\Phi_1^m, \ldots, \Phi_K^m]\,.
\nonumber
\eeqan
Again for fixed deformations $(\bar \Phi^m_k)_{\substack{k=1,\ldots, K,\\ m=1,\ldots, M}}$ and skipping the non negativity constraint for the densities one obtains a system of linear equations to be solved for $(\bar \U^m_k)_{\substack{k=0,\ldots, K-1,\\ m=1,\ldots, M}}$ with $\bar \U^1_0 = \ldots = \bar \U^M_0 = \bar \U^\bary$.
This linear system consists of $M$ copies of the equations for $(\bar \U_1,\ldots, \bar \U_{K-1})$ in the system, where we replace $\bar \U_k$ by $\bar \U_k^m$ and 
$\Phi_k$ by $\Phi_k^m$, and an additional set of equations for $\bar \U^\bary$, \ie
\beqn\label{eq:ELbaryh}
\mass_h[\Id,\Id] \bar \U^\bary = \sum_{m=1}^M \lambda^m \left( \mass_h[\det( D \Phi_1^m ),\Phi^m_1,\Id]\bar \U_1^m - \frac{\penaltyPushforward}{2} \mass_h[|\Phi_1^m - \id|^2,\Id,\Id] \bar 1  \right)\,.
\nonumber
\eeqn
Still, a slightly modified strict convexity argument proves that the energy $\EdgKh$ is strictly convex in the images $\U_k^m$ for $k=1,\ldots, K$, $m=1,\ldots, M$ and in the 
additional  image $\U^\bary$. In particular, there exists a unique solution of the linear system. Let us remark that this is no longer clear if we replace $\WdgKh[\U^m,\U]$ by
$\WdgKh[\U,\U^m]$ in the definition of the fully discrete barycenter in \eqref{eq:bardiscrete}.
In the implementation, we apply an analogous alternating descent scheme as described in Section \ref{sec:SpatialDiscretization} to compute fully discrete approximations of the weighted Riemannian barycenter.
Furthermore, we use a cascadic approach, starting with coarse time discretizations and then successively refining the discretization in time.
Figure \ref{fig:beetsBarycenterEnergy} shows barycenters (with equal weights $\lambda = \frac1M$) for three different sets of sugar beet slices extracted from noninvasive 3D MRI images at different 
days after plantation for different viscous dissipation parameters $\gamma$. Furthermore, we show the variability of the different contributions to the path energy between the barycenter and the input images for all input sugar beets. Finally, we display in Figure~\ref{fig:woodBarycentreForward} weighted barycenters of three different wood textures with all admissible combinations of $\lambda_m \in \{0,\, \frac13,\, \frac23, \,1\}$. 
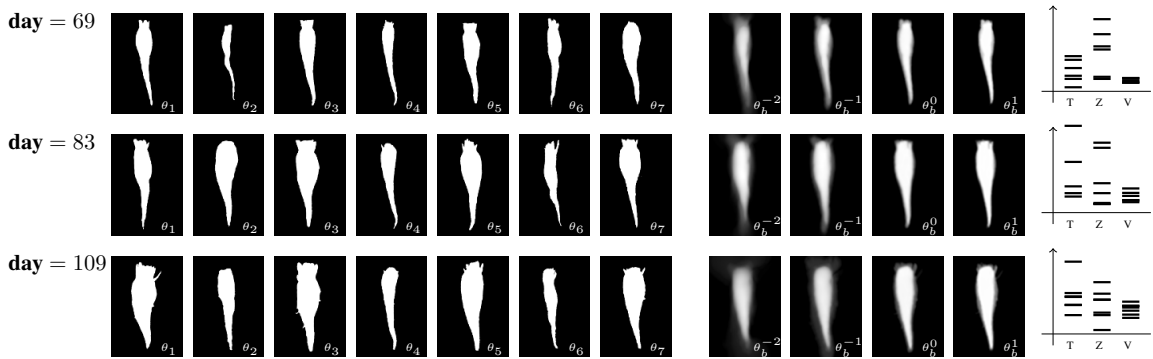
\begin{figure}[htbp!]
\centering
 \resizebox{0.85\linewidth}{!}{
  \begin{minipage}[h]{1.0\textwidth} 
    \setlength{\unitlength}{.05\linewidth}
      \begin{picture}(17,2.2)
                  \put(-2.1, 1.6){ \color{black}{ \bf day $=69$ \bf} }
      \put(0,0){\includegraphics[width=0.065\textwidth]{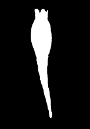}}
      \put(0.7, 0.1){ \color{white}{ \tiny $\u_1$ \tiny} }
      \put(1.5,0){\includegraphics[width=0.065\textwidth]{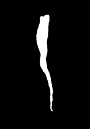}}
      \put(2.2, 0.1){ \color{white}{ \tiny $\u_2$ \tiny} }
      \put(3,0){\includegraphics[width=0.065\textwidth]{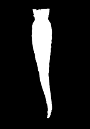}}
      \put(3.7, 0.1){ \color{white}{ \tiny $\u_3$ \tiny} }
      \put(4.5,0){\includegraphics[width=0.065\textwidth]{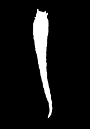}}
      \put(5.2, 0.1){ \color{white}{ \tiny $\u_4$ \tiny} }
      \put(6, 0){\includegraphics[width=0.065\textwidth]{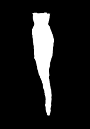}}
      \put(6.7, 0.1){ \color{white}{ \tiny $\u_5$ \tiny} }
      \put(7.5,0){\includegraphics[width=0.065\textwidth]{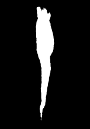}}
      \put(8.2, 0.1){ \color{white}{ \tiny $\u_6$ \tiny} }
      \put(9,0){\includegraphics[width=0.065\textwidth]{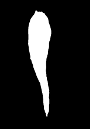}}
      \put(9.7, 0.1){ \color{white}{ \tiny $\u_7$ \tiny} }
      \put(11,0){\includegraphics[width=0.065\textwidth]{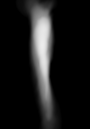}}
      \put(11.6, 0.1){ \color{white}{ \tiny $\u_b^{-2}$ } \tiny }
      \put(12.5,0){\includegraphics[width=0.065\textwidth]{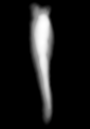}}
      \put(13.1, 0.1){ \color{white}{ \tiny $\u_b^{-1}$ \tiny} }
      \put(14,0){\includegraphics[width=0.065\textwidth]{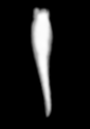}}
      \put(14.7, 0.1){ \color{white}{ \tiny $\u_b^0$ \tiny} }
      \put(15.5,0){\includegraphics[width=0.065\textwidth]{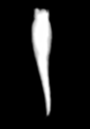}}
      \put(16.2, 0.1){ \color{white}{ \tiny $\u_b^1$ \tiny} }
    \end{picture}
    \begin{tikzpicture}[scale=0.9]
      \draw[->] (0.3,-0.2) -- (0.3,1.5);

      \draw(0.6+0.5*0,-0.2) node{\tiny T};
      \draw(0.6+0.5*1,-0.2) node{\tiny Z};
      \draw(0.6+0.5*2,-0.2) node{\tiny V};

      \foreach \v/\w in {
      0/0.00181229,       1/0.00145615, 
      2/0.00101358,
                        0/0.000469023,       1/0.00525616,
      2/0.00100055,
      0/0.00145031,       1/0.006667,
      2/0.00129068,
      0/0.00271816,       1/0.00488122,
      2/0.00112419,
      0/0.00367197,       1/0.00169383,
      2/0.00156182,
      0/0.00410892,       1/0.00842908,
      2/0.0015314
      }
      {
        \draw[line width=1pt]  (0.5 + 0.5*\v,150*\w) -- (0.5 + 0.5*\v+0.3,150*\w);
      };
      
      \draw[-] (0.1,0) -- (2.0,0);
      
    \end{tikzpicture}
  \end{minipage}
  }
  
  \resizebox{0.85\linewidth}{!}{
  \begin{minipage}[h]{1.0\textwidth} 
    \setlength{\unitlength}{.05\linewidth}
      \begin{picture}(17,2.2)
                  \put(-2.1, 1.6){ \color{black}{ \bf day $= 83$ \bf} }
      \put(0,0){\includegraphics[width=0.065\textwidth]{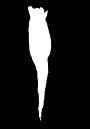}}
      \put(0.7, 0.1){ \color{white}{ \tiny $\u_1$ \tiny} }
      \put(1.5,0){\includegraphics[width=0.065\textwidth]{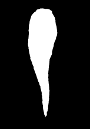}}
      \put(2.2, 0.1){ \color{white}{ \tiny $\u_2$ \tiny} }
      \put(3,0){\includegraphics[width=0.065\textwidth]{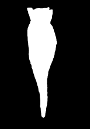}}
      \put(3.7, 0.1){ \color{white}{ \tiny $\u_3$ \tiny} }
      \put(4.5,0){\includegraphics[width=0.065\textwidth]{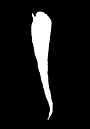}}
      \put(5.2, 0.1){ \color{white}{ \tiny $\u_4$ \tiny} }
      \put(6, 0){\includegraphics[width=0.065\textwidth]{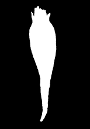}}
      \put(6.7, 0.1){ \color{white}{ \tiny $\u_5$ \tiny} }
      \put(7.5,0){\includegraphics[width=0.065\textwidth]{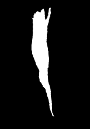}}
      \put(8.2, 0.1){ \color{white}{ \tiny $\u_6$ \tiny} }
      \put(9,0){\includegraphics[width=0.065\textwidth]{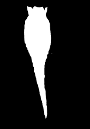}}
      \put(9.7, 0.1){ \color{white}{ \tiny $\u_7$ \tiny} }
      \put(11,0){\includegraphics[width=0.065\textwidth]{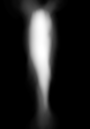}}
      \put(11.6, 0.1){ \color{white}{ \tiny $\u_b^{-2}$ } \tiny }
      \put(12.5,0){\includegraphics[width=0.065\textwidth]{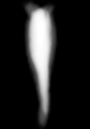}}
      \put(13.1, 0.1){ \color{white}{ \tiny $\u_b^{-1}$ \tiny} }
      \put(14,0){\includegraphics[width=0.065\textwidth]{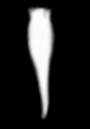}}
      \put(14.7, 0.1){ \color{white}{ \tiny $\u_b^0$ \tiny} }
      \put(15.5,0){\includegraphics[width=0.065\textwidth]{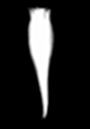}}
      \put(16.2, 0.1){ \color{white}{ \tiny $\u_b^1$ \tiny} }
    \end{picture}
    \begin{tikzpicture}[scale=0.9]
      \draw[->] (0.3,-0.2) -- (0.3,1.5);

      \draw(0.6+0.5*0,-0.2) node{\tiny T};
      \draw(0.6+0.5*1,-0.2) node{\tiny Z};
      \draw(0.6+0.5*2,-0.2) node{\tiny V};

      \foreach \v/\w in {
            0/0.00226764,       1/0.00100554,
      2/0.00122788,
      0/0.00305955,       1/0.003439315,
      2/0.00146852,
      0/0.00188028,       1/0.00756415,
      2/0.00194645
      0/0.00469654,       1/0.0081694,
      2/0.00229808,
      0/0.00591884,       1/0.002276935
      2/0.00140606
      0/0.0046859,       1/0.001007435,
      2/0.00282609,
      0/0.010133,       1/0.00113646,
      2/0.00143406
      }
      {
        \draw[line width=1pt]  (0.5 + 0.5*\v,150*\w) -- (0.5 + 0.5*\v+0.3,150*\w);
      };
      
      \draw[-] (0.1,0) -- (2.0,0);
    \end{tikzpicture}
  \end{minipage}
  }
  
  \resizebox{0.85\linewidth}{!}{
  \begin{minipage}[h]{1.0\textwidth} 
    \setlength{\unitlength}{.05\linewidth}
      \begin{picture}(17,2.2)
                  \put(-2.1, 1.6){ \color{black}{ \bf day $=109$ \bf} }
      \put(0,0){\includegraphics[width=0.065\textwidth]{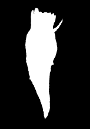}}
      \put(0.7, 0.1){ \color{white}{ \tiny $\u_1$ \tiny} }
      \put(1.5,0){\includegraphics[width=0.065\textwidth]{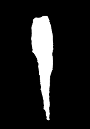}}
      \put(2.2, 0.1){ \color{white}{ \tiny $\u_2$ \tiny} }
      \put(3,0){\includegraphics[width=0.065\textwidth]{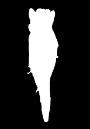}}
      \put(3.7, 0.1){ \color{white}{ \tiny $\u_3$ \tiny} }
      \put(4.5,0){\includegraphics[width=0.065\textwidth]{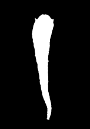}}
      \put(5.2, 0.1){ \color{white}{ \tiny $\u_4$ \tiny} }
      \put(6, 0){\includegraphics[width=0.065\textwidth]{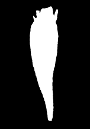}}
      \put(6.7, 0.1){ \color{white}{ \tiny $\u_5$ \tiny} }
      \put(7.5,0){\includegraphics[width=0.065\textwidth]{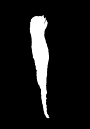}}
      \put(8.2, 0.1){ \color{white}{ \tiny $\u_6$ \tiny} }
      \put(9,0){\includegraphics[width=0.065\textwidth]{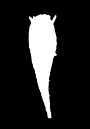}}
      \put(9.7, 0.1){ \color{white}{ \tiny $\u_7$ \tiny} }
      \put(11,0){\includegraphics[width=0.065\textwidth]{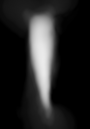}}
      \put(11.6, 0.1){ \color{white}{ \tiny $\u_b^{-2}$ } \tiny }
      \put(12.5,0){\includegraphics[width=0.065\textwidth]{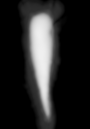}}
      \put(13.1, 0.1){ \color{white}{ \tiny $\u_b^{-1}$ \tiny} }
      \put(14,0){\includegraphics[width=0.065\textwidth]{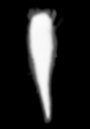}}
      \put(14.7, 0.1){ \color{white}{ \tiny $\u_b^0$ \tiny} }
      \put(15.5,0){\includegraphics[width=0.065\textwidth]{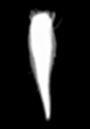}}
      \put(16.2, 0.1){ \color{white}{ \tiny $\u_b^1$ \tiny} }
    \end{picture}
    \begin{tikzpicture}[scale=0.9]
      \draw[->] (0.3,-0.2) -- (0.3,1.5);

      \draw(0.6+0.5*0,-0.2) node{\tiny T};
      \draw(0.6+0.5*1,-0.2) node{\tiny Z};
      \draw(0.6+0.5*2,-0.2) node{\tiny V};

      \foreach \v/\w in {
            0/0.00432272,       1/0.004,
      2/0.00377505,
      0/0.00438524,       1/0.00219678,
      2/0.00187723,
      0/0.00219663,       1/0.00399182,
      2/0.00329378,
      0/0.003371,       1/0.00470378,
      2/0.00271384,
      0/0.00843153,       1/0.00248174,
      2/0.00308222,
      0/0.00478247,       1/0.00603610,
      2/0.00326923,
      0/0.0033997,       1/0.000464624,
      2/0.00226188
      }
      {
        \draw[line width=1pt]  (0.5 + 0.5*\v,150*\w) -- (0.5 + 0.5*\v+0.3,150*\w);
      };
      
      \draw[-] (0.1,0) -- (2.0,0);
      
    \end{tikzpicture}
  \end{minipage}
  }
  \caption{Barycentres of different sets of sugar beet slices (left) for $\penaltyPushforward = 10^{-1}$ and for $\gamma=10^{-2},\, 10^{-1},\, 1,\,10$ (from left to right with $\u_b^j$ corresponding to $\gamma = 10^{j}$). 
  On the right the transport cost (T), density modulation cost (Z), and viscous dissipation cost (V) are plotted for all input slices in the case $\gamma=1$ (cost values for the second beet at $day 69$ are excluded as outliers). (data provided by research network CROP.SENSe.net)}
\label{fig:beetsBarycenterEnergy}
\end{figure} 

\begin{figure}[htbp!]
\setlength{\unitlength}{0.06\textwidth}
  \resizebox{0.85\linewidth}{!}{
    \begin{picture}(22,11.5)
            \put(4.5, 9){\includegraphics[width=0.17\textwidth]{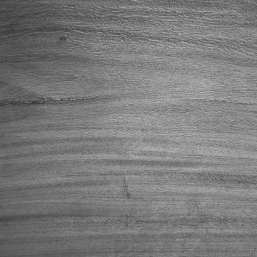}}       \put(4.5, 9.2){ \color{white}{ \bf (0,0,1) \bf} }
      \put(0,0){\includegraphics[width=0.17\textwidth]{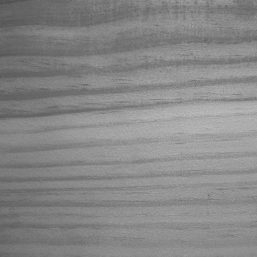}}       \put(0,0.2){ \color{white}{ \bf (1,0,0) \bf} }
      \put(9,0){\includegraphics[width=0.17\textwidth]{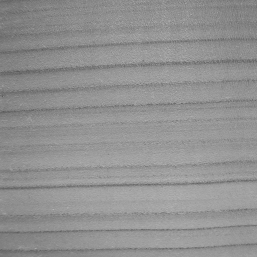}}       \put(9,0.2){ \color{white}{ \bf (0,1,0) \bf} }
      \put(4.5, 3){\includegraphics[width=0.17\textwidth]{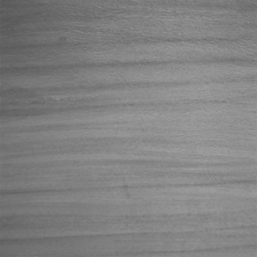}}       \put(4.5, 3.2){ \color{white}{ \bf ($\frac{1}{3}$,$\frac{1}{3}$,$\frac{1}{3}$) \bf} }
      \put(1.5, 3){\includegraphics[width=0.17\textwidth]{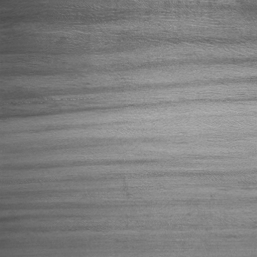}}
      \put(1.5, 3.2){ \color{white}{ \bf ($\frac{1}{3}$,0,$\frac{2}{3}$) \bf} }
      \put(3.0, 6){\includegraphics[width=0.17\textwidth]{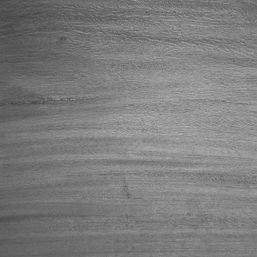}}
      \put(3.0, 6.2){ \color{white}{ \bf ($\frac{2}{3}$,0,$\frac{1}{3}$) \bf} }
      \put(3.0,0){\includegraphics[width=0.17\textwidth]{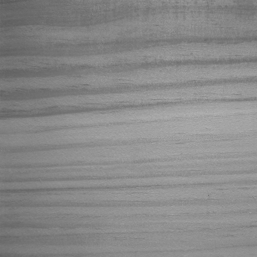}}
      \put(3.0, 0.2){ \color{white}{ \bf ($\frac{2}{3}$,$\frac{1}{3}$,0) \bf} }
      \put(6,0){\includegraphics[width=0.17\textwidth]{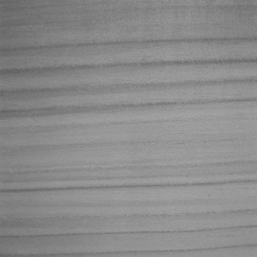}}
      \put(6, 0.2){ \color{white}{ \bf ($\frac{1}{3}$,$\frac{2}{3}$,0) \bf} }
      \put(6, 6){\includegraphics[width=0.17\textwidth]{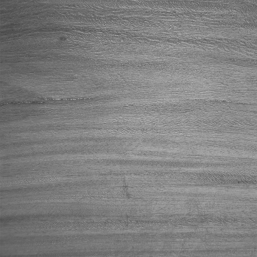}}
      \put(6, 6.2){ \color{white}{ \bf (0,$\frac{1}{3}$,$\frac{2}{3}$) \bf} }
      \put(7.5, 3){\includegraphics[width=0.17\textwidth]{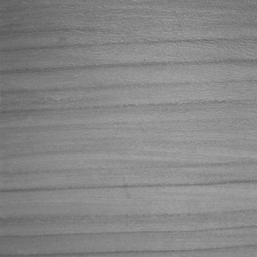}}
      \put(7.5, 3.2){ \color{white}{ \bf (0,$\frac{2}{3}$,$\frac{1}{3}$) \bf} }
      
            \put(13.25, 9.5){\includegraphics[width=0.13\textwidth]{Images/TriangleBarycenterWoodForward/woodBarycenter.png}}
      \put(13.25, 9.7){ \color{white}{ \bf ($\frac{1}{3}$,$\frac{1}{3}$,$\frac{1}{3}$) \bf} }
      \put(13.25, 7.125){\includegraphics[width=0.13\textwidth]{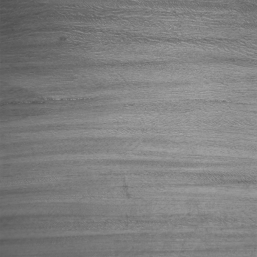}} 
      \put(13.25, 4.75){\includegraphics[width=0.13\textwidth]{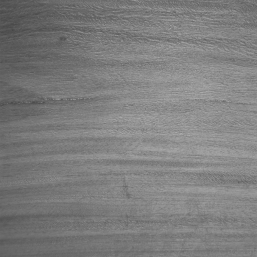}}
      \put(13.25, 2.375){\includegraphics[width=0.13\textwidth]{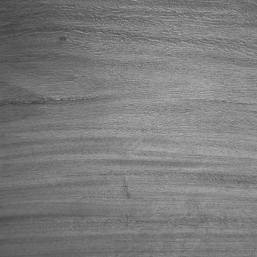}}
      \put(13.25, 0){\includegraphics[width=0.13\textwidth]{Images/TriangleBarycenterWoodForward/wood1.png}}
      \put(13.25, 0.2){ \color{white}{ \bf (0,0,1) \bf} }
      \put(16, 9.5){\includegraphics[width=0.13\textwidth]{Images/TriangleBarycenterWoodForward/woodBarycenter.png}}
      \put(16, 9.7){ \color{white}{ \bf ($\frac{1}{3}$,$\frac{1}{3}$,$\frac{1}{3}$) \bf} }
      \put(16, 7.125){\includegraphics[width=0.13\textwidth]{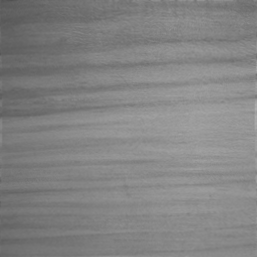}} 
      \put(16, 4.75){\includegraphics[width=0.13\textwidth]{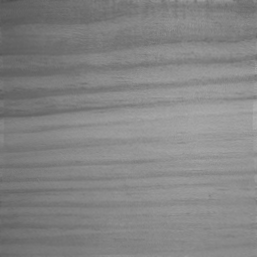}} 
      \put(16, 2.375){\includegraphics[width=0.13\textwidth]{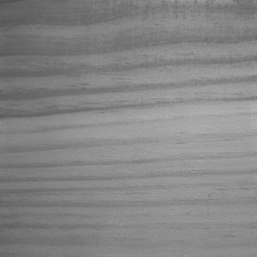}} 
      \put(16, 0){\includegraphics[width=0.13\textwidth]{Images/TriangleBarycenterWoodForward/wood2.png}} 
      \put(16, 0.2){ \color{white}{ \bf (1,0,0) \bf} }
      \put(18.75, 9.5){\includegraphics[width=0.13\textwidth]{Images/TriangleBarycenterWoodForward/woodBarycenter.png}}
      \put(18.75, 9.7){ \color{white}{ \bf ($\frac{1}{3}$,$\frac{1}{3}$,$\frac{1}{3}$) \bf} }
      \put(18.75, 7.125){\includegraphics[width=0.13\textwidth]{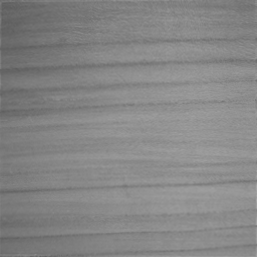}} 
      \put(18.75, 4.75){\includegraphics[width=0.13\textwidth]{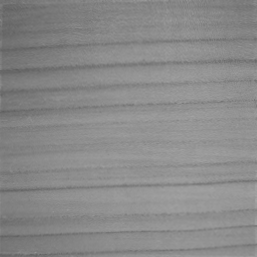}}
      \put(18.75, 2.375){\includegraphics[width=0.13\textwidth]{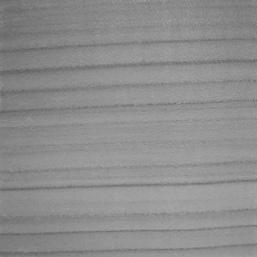}}
      \put(18.75, 0){\includegraphics[width=0.13\textwidth]{Images/TriangleBarycenterWoodForward/wood3.png}}
      \put(18.75, 0.2){ \color{white}{ \bf (0,1,0) \bf} }
    \end{picture}
  }
\caption{Weighted barycentres of three different wood textures (http://de.wikipedia.org/wiki/Holz) are shown for $\penaltyPushforward = 10^{-1}$, $\gamma = 1$ (Left: barycentric triangle
where $(\lambda_0, \lambda_1, \lambda_2)$ are overlaid each texture, right: discrete geodesics for $K=4$ between the three input textures and the barycentre are depicted.}
\label{fig:woodBarycentreForward} 
\end{figure}

\FloatBarrier
\section{Conclusion and Outlook}
In this paper we have developed a combined optimal transport and metamorphosis model and propose an effective time discretization 
of the path energy in the space of density maps. The method allows us to approximate the original Wasserstein distance and for larger viscosity parameter interesting additional effects can be observed. In particular in applications to images the incorporated source term turns out to be an appropriate way to deal with mass variability. 
Let us briefly comment on limitations and possible future extensions of the model.
So far, in the non-viscous case the source term has to be absolutely continuous with respect to the Lebesque measure (cf. Section \ref{sec:existenceNonViscous}), because the measure $\Leb^\perp$ in the decomposition of the source measures is not unique. Therefore, a singular part in $L^2$ would depend on the decomposition.
An alternative model with a source term in $L^1$ including singular parts is work in progress. In addition, in the time discrete model discussed in Section \ref{sec:timediscrete} the treatment of the source term in $L^2$ required special care, since we aim at measuring the change of densities, which is an $L^1$-concept.
Furthermore, for our generalized model including dissipation existence of geodesics in the time continuous case is unclear. In the non-viscous case we made use of a change of variables by considering the momentum instead of the velocity, but for the viscous dissipation term this does not appear to be the appropriate concept. This also renders the verification of $\Gamma$-convergence more difficult than in the case of the metamorphosis model in \cite{BeEf14}.

\section*{Acknowledgement}
The authors acknowledge support of the Collaborative Research Centre 1060 funded by the German Science foundation. This work is further supported by the King Abdullah University for Science and Technology (KAUST) Award No. KUK-I1-007-43 and the EPSRC grant Nr. EP/M00483X/1.

\end{document}